\documentclass{amsart}
\usepackage{amssymb}
\usepackage{color}
\usepackage[allcolors = blue, colorlinks = true]{hyperref}
\usepackage{mathptmx}
\usepackage{bm}
\usepackage{amsrefs}
\usepackage{anyfontsize}
\usepackage{comment}
\usepackage{mathtools}
\usepackage{esint}
\usepackage{tikz}
\usepackage{resmes}
\usepackage{bbm}
\usepackage[
  hmarginratio={1:1},     
  vmarginratio={1:1},     
  textwidth=450pt,       
  heightrounded,        
]{geometry}
\newtheorem{theorem}{Theorem}[section]
\newtheorem{lemma}[theorem]{Lemma}
\newtheorem{corollary}[theorem]{Corollary}
\newtheorem{proposition}[theorem]{Proposition}
\theoremstyle{definition}
\newtheorem{definition}[theorem]{Definition}
\newtheorem{remark}[theorem]{Remark}
\newtheorem{example}[theorem]{Example}
\usepackage[cal = euler]{mathalpha}
\newcommand{\N}{\mathbb{N}}
\newcommand{\R}{\mathbb{R}}
\newcommand{\Rn}{\R^n}
\newcommand{\A}{\mathcal{A}}
\newcommand{\B}{\mathcal{B}}
\newcommand{\M}{\mathcal{M}}
\newcommand{\Lag}{\mathcal{L}}
\newcommand{\LB}{\mathfrak{L}}
\newcommand{\Bor}{\mathfrak{B}}
\newcommand{\Om}{\Omega}
\newcommand{\la}{\langle}
\newcommand{\ra}{\rangle}
\newcommand{\loc}{\textnormal{loc}}
\DeclareMathOperator{\Curv}{Curv}
\DeclareMathOperator{\Mod}{Mod}
\numberwithin{equation}{section}

\AtBeginDocument{%
   \def\MR#1{}
}
\begin{document}
\title{Duality for the gradient of a $p$-harmonic function and the existence of gradient curves}
\author{Sylvester Eriksson-Bique}
\address[Sylvester Eriksson-Bique]{Department of Mathematics and Statistics, University of
Jyv\"askyl\"a, PO~Box~35, FI-40014 Jyv\"askyl\"a, Finland}
\email{sylvester.d.eriksson-bique@jyu.fi}
\author{Saara Sarsa}
\address[Saara Sarsa]{Department of Mathematics and Statistics, University of
Jyv\"askyl\"a, PO~Box~35, FI-40014 Jyv\"askyl\"a, Finland}
\email{saara.m.sarsa@jyu.fi}

\begin{abstract}
Every convex optimization problem has a dual problem. The $p$-Dirichlet problem in metric measure spaces is an optimization problem whose solutions are $p$-harmonic functions. What is its dual problem? In this paper, we give an answer to this problem in the following form. We give a generalized modulus problem whose solution is the gradient of the $p$-harmonic function for metric measure spaces. Its dual problem is an optimization problem for measures on curves and we show exact duality and the existence of minimizers for this dual problem under appropriate assumptions. When applied to $p$-harmonic functions the minimizers of this dual problem are supported on gradient curves, yielding a natural concept associated to such functions that has yet to be studied. This process defines a natural dual metric current and proves the existence of gradient curves. These insights are then used to construct a counter example answering the old ``sheaf problem'' on metric spaces: in contrast to Euclidean spaces, in general metric spaces being $p$-harmonic is not strictly speaking a local property. 
\end{abstract}
\thanks{The first author was supported by Research Council of Finland grants 354241 and 356861. The second author was supported by Eemil Aaltonen Foundation and Jenny and Arttu Wihuri foundation through the research group ``Quasiworld network'', as well as by the Research Council of Finland grant 354241. We thank Nageswari Shanmugalingam, Xiaodan Zhou, Qiu Ling, Anders Björn and Jana Björn for helpful discussions. The failure of the sheaf property was related closely to a separate project of the first author with Anders Björn, Jana Björn and Xiaodan Zhou. Part of this work was done while the first author visited Okinawa Institute of Science and Technology as part of the TSVP program.}
\maketitle
\tableofcontents
\section{Introduction}

\subsection{Outline}
In this article we study a natural dual problem associated to $p$-harmonic functions and the $p$-Dirichlet problem. Our motivation for this result is to prove that a $p$-harmonic function on a metric measure space has gradient curves. The solution of the dual problem yields an intriguing new concept: a measure on the space of curves that is supported on the gradient curves. This measure is connected to several concepts, such as Alberti representations, currents and $p$-harmonic measures, that have recently been actively researched. The insights gained from this study lead us to an example that shows that sheaf property of $p$-harmonic functions (an open question posed in \cite{Bjorn2011}*{Section 9.4}) fails in the setting of metric measure spaces.

\subsection{Gradient curves}
To motivate our work, let us consider the Euclidean setting, where the connection between the behavior of a function $u$ and the geometry of its gradient curves is more evident. Our extension to metric measure spaces is motivated by the desire to establish a similar geometric structure for general $p$-harmonic functions. 

A gradient curve of a sufficiently smooth function $u\colon\Rn\to\R$ is a $C^1$ curve $\gamma\colon[a,b]\to\Rn$, $a<b$, such that
\begin{equation} \label{eq:intro-gradient-flow-eq}
    \gamma'(t)=\frac{Du(\gamma(t))}{|Du(\gamma(t))|}
    \quad\text{for all }t\in(a,b).
\end{equation}
Such curves are sometimes called gradient flows.
Note that the gradient curves in our setting are normalized: $|\gamma'|\equiv 1$. Gradient curves capture important geometric information of the function $u$ and of the ambient space. Indeed, they have been applied to the study of Alexandrov spaces \cites{Petrunin, Perelman}, and they play a key role in optimal transport and in the study of Sobolev spaces \cites{Ambrosio2015, Villani, Gigli2015,AGS}. Despite the simplicity of their definition, their regularity properties are not obvious even in the Euclidean case. A simple example, such as $u(x,y)=x^2-y^2$, shows that gradient curves of a smooth harmonic function do not necessarily yield a smooth foliation of the domain due to the presence of critical points. Due to this difficulty, it is not at all clear, if such curves could be found in a general metric measure setting. Indeed, usually the existence of gradient flows involves curvature assumptions or some convexity \cites{Petrunin, Perelman, Ambrosio2005, Villani}. Without such assumptions, there are very few methods available. 
 
In Euclidean spaces, there are special tools to obtain many gradient curves. If $u$ is a harmonic function on a planar domain ($n=2$), then the gradient curves of $u$ are precisely the level sets of its conjugate harmonic function $v$. This follows from the Cauchy-Riemann equations. 
An analogous result holds if $u$ is a planar $p$-harmonic function, that is, a weak solution of the $p$-Laplace equation
\begin{equation} \label{eq:intro-p-Laplace}
    \operatorname{div}(|Du|^{p-2}Du)=0,
\end{equation}
see \cite{Aronsson1988}.
In this case $u$ and its conjugate function $v$ are connected via
$$ \frac{\partial v}{\partial x}=-|Du|^{p-2}\frac{\partial u}{\partial y}
\quad\text{and}\quad\frac{
\partial v}{\partial y}=|Du|^{p-2}\frac{\partial u}{\partial x}. $$
In particular, their gradients are perpendicular and $v$ is $q$-harmonic with $q=\frac{p}{p-1}$.

In higher dimensions, the conjugate function $v$ can be replaced by a conjugate form. This is easier to see via a fluid dynamic interpretation of the conjugate function, which persists to higher dimensions. Indeed, in all dimensions, gradient curves can be seen as reparametrizations of stream curves of the flow of the divergence free vector field $|Du|^{p-2}Du$ -- at least, when this vector field is sufficiently regular. In the plane $n=2$, the conjugate function $v$ corresponds to the stream function for this flow, which measures how much flow passes through a curve connecting two points. In higher dimensions, the stream function is replaced by a $(n-2)$-form $\alpha$ which satisfies $d\alpha = \ast (|du|^{p-2}du)$, where $\ast$ is the Hodge dual. The existence of such a form $\alpha$ is implied by \eqref{eq:intro-p-Laplace}, which is equivalent to $\ast (|du|^{p-2}du)$ being a closed form. If $\alpha$ is integrated over an $n-2$-surface $\partial \sigma$ enclosing $\sigma$, then one obtains the total flow passing through $\sigma$. 

In summary, gradient curves are connected to conjugate functions and divergence free flows. Further, the constructions in the previous two paragraphs show, that it is natural to seek structures, such as forms, capturing many gradient curves instead of just solving for individual curves. In this article, we aim to construct gradient curves for $p$-harmonic functions on general metric measure spaces $X=(X,d,\mu)$ which also capture dual relationships. However, the tools from the previous paragraphs are not at our disposal in this general setting and we need to find replacements. 

Harmonic functions in the setting of a metric measure space were originally studied by Cheeger \cite{Cheeger1999} and Shanmugalingam \cites{Shanmugalingam2001}, with a different approach introduced by Gigli in \cite{Gigli2015}. By now, potential theory in metric spaces has been quite well developed; see the surveys \cites{AMS,Heinonen2007}. For a detailed exposition for the established theory of first order calculus in metric measure spaces we refer to the monograph by Heinonen, Koskela, Tyson and Shanmugalingam \cite{Heinonen2015} and Heinonen \cite{Heinonen2001}. 
For details of the potential theory in metric measure spaces we refer to the monograph by Björn and Björn \cite{Bjorn2011}. For classical non-linear potential theory, see the monograph of Heinonen, Kilpeläinen and Martio \cite{HKM}.

The standard assumptions for validity of reasonable potential theory in a metric measure space are that $X$ is complete, doubling and satisfies $p$-Poincar\'{e} inequality \cites{Heinonen2015,Bjorn2011}. Since we do not use the $p$-Poincar\'e inequality directly in our proofs, we omit its definition; see e.g. \cite{Heinonen2015}. With these assumptions, neither the system \eqref{eq:intro-gradient-flow-eq} of equations for gradient curves nor the $p$-Laplace equation \eqref{eq:intro-p-Laplace} makes sense when the Euclidean space $\Rn$ is replaced with a metric measure space $X$. 
Instead, $p$-harmonic functions are defined as minimizers of $p$-energy functional. If $\Om\subset X$ is a domain, $u\in N^{1,p}(\Om)$ is $p$-harmonic if it is a continuous local minimizer of the $p$-energy functional
$$ v\mapsto\int_{\Om}g_v^pd\mu $$
among $v\in N^{1,p}(\Om)$. If $X$ satisfies a $p$-Poincar\'e inequality, it is easy to see that there exist many such functions, see Section \ref{sec:energy}. 
Here $g_v$ denotes the minimal $p$-weak upper gradient of $v$. (For details of the definitions, see Section \ref{sec:preliminaries}.)
By the definition of $p$-weak upper gradient,
\begin{equation} \label{eq:intro-upper-gradient}
    |u(x)-u(y)|\leq\int_\gamma g_uds
\end{equation}
for $p$-a.e. nonconstant rectifiable curve $\gamma$ on $\Om$ with endpoints $x,y\in \Om$.
If a simple (nonconstant) rectifiable curve $\gamma$ satisfies
\begin{equation} \label{eq:intro-gradient-curves-def}
    |u(x)-u(y)|=\int_{\gamma|_{[x,y]}} g_uds,
\end{equation}
for all $x,y\in \gamma$ then we say that $\gamma$ is a gradient curve of $u$. Here $\gamma|_{[x,y]}$ is the subcurve connecting $x$ to $y$.
\begin{remark}
    It is enough to check \eqref{eq:intro-gradient-curves-def} for the end points of $\gamma:[0,\ell_\gamma]\to X$, i.e. $x=\gamma(0), y=\gamma(\ell_\gamma)$. Indeed, if $\gamma$ satisfies \eqref{eq:intro-upper-gradient} and this end point condition, then \eqref{eq:intro-gradient-curves-def} also holds simply by the triangle inequality. See Remark \ref{rmk:gradientcurve} for a more detailed argument.
\end{remark}
Note that this definition is compatible with \eqref{eq:intro-gradient-flow-eq}. It is direct to verify, that if $u\in C^1(\R^n)$, then $g_u=|\nabla u|$ and \eqref{eq:intro-gradient-curves-def} coincides with \eqref{eq:intro-gradient-flow-eq}. Our main result gives a natural measure supported on gradient curves of $u$, and thus the existence of such curves. In the statement of the theorem, we let $\Curv(\overline{\Om})$ denote the set of all rectifiable, arc-length parametrized curves on $\overline{\Om}$, equipped with a suitable metric. (For details, see Section \ref{sec:preliminaries}.)

\begin{theorem} \label{thm:intro-main}
Let $X=(X,d,\mu)$ be a complete, doubling $p$-Poincar\'{e} space, $1<p<\infty$.
Suppose that $\Om'\subset X$ is a domain and $u\in N^{1,p}_{\loc}(\Om')$ a $p$-harmonic function. 
Let us fix a bounded subdomain $\Om\subset\subset\Om'$ such that $\mu(\partial\Om)=0$ and $u$ is non-constant in $\Om$. 
There exists a Borel measure $\eta^*$ on
\begin{equation} \label{eq:intro-Gamma}
    \Gamma:=\{\gamma\in\Curv(\overline{\Om}):\gamma(0),\gamma(\ell_\gamma)\in\partial\Om\}
\end{equation}
such that the following hold:
\begin{enumerate}
    \item The $p$-energy of $u$ is given by integration over the gradient curves of $u$ with respect to $\eta^*$, 
    $$\|g_u\|_{L^p(\Om)}=\int_\Gamma|u(\gamma(0))-u(\gamma(\ell_\gamma))|d\eta^*. $$
    \item The measure $\eta^*$ is supported on gradient curves of $u$, that is, on those curves $\gamma\in\Gamma$ for which
    $$ |u(\gamma(0))-u(\gamma(\ell_\gamma))|=\int_\gamma g_uds. $$
    \item For any Borel set $E\subset\overline{\Om}$, the following identity holds:
    $$ \int_\Gamma\int_\gamma\mathbbm{1}_Edsd\eta^*=\int_{E}\Big(\frac{g_u}{\|g_u\|_{L^p(\Om)}}\Big)^{p-1}d\mu. $$
    In particular, for $\mu$-a.e. $x\in\{y\in\Om:g_u(y)\neq 0\}$ there exists a gradient curve $\gamma$ of $u$ passing through $x$.
\end{enumerate}
\end{theorem}

This theorem immediately implies the existence of gradient curves that pass through a.e. point of the space where the gradient does not vanish.

\begin{corollary}
Let $X=(X,d,\mu)$ be a complete, doubling $p$-Poincar\'{e} space, $1<p<\infty$.
Suppose that $\Om'\subset X$ be a domain and $u\in N^{1,p}_{\loc}(\Om')$ a non-constant $p$-harmonic function. Then for $\mu$-a.e. $x\in\{y\in\Om':g_u(y)\neq 0\}$ there exists a gradient curve $\gamma$ of $u$ that passes through $x$.
\end{corollary}

\begin{remark}
    The $p$-Poincar\'e inequality in these statements is only used to guarantee the existence of a continuous solution to the $p$-Dirichlet problem.  Whenever there is some other way, that avoids the $p$-Poincar\'e inequality, and produces (continuous) $p$-harmonic functions, this assumption could be avoided. There are natural settings where this is applicable, i.e. where continuous $p$-harmonic functions are obtained without a $p$-Poincar\'e inequality. For example, in the work of Rajala on uniformization \cite{Rajala}, his capacity minimizers are shown to be continuous. Our results would construct in those settings a (dual) measure supported on curves. This measure should be closely connected to the conjugate function constructed in \cite{Rajala}. Studying this relationship may lead new perspectives into unformization.

    In Theorem \ref{thm:intro-main}, we assume that $\Omega$ is bounded. It seems likely that one could modify the statement to obtain an unbounded result and prove it by exhausting a domain by bounded subdomains. This is a bit subtle and is not studied here. The measure $\eta^*$ may need not be finite; see Example \ref{ex:non-Radonness}. This is a small technical point, which can be disregarded at first reading. The proof shows that $\eta^*$ is always $\sigma$-finite. The measure $\eta^*$ can also be non-unique, and this phenomenon is exploited in the counter-example to the sheaf property.
\end{remark}

\begin{example} A simple example illustrates the construction of the measure $\eta^*$ quite well.
Let $X=\R^2$ with the usual Euclidean metric and Lebesgue measure and let $u(x,y)=x$.
Then $u$ is $p$-harmonic in $\R^2$ and $g_u=1$. Let $\Om=(0,a)\times(0,b)$. The gradient curves in $\Om$ are straight horizontal lines, $\eta^*$ is a uniform measure on such curves and 
$$ \|g_u\|_{L^p(\Om)}=(ab)^{1/p} \quad\text{and}\quad
\int_\Gamma|u(\gamma(0))-u(\gamma(\ell_\gamma))|d\eta^*=a\eta^*(\Gamma). $$
We can apply Theorem \ref{thm:intro-main} to conclude that $\eta^*(\Gamma)=\frac{b^{1/p}}{a^{1/q}}$ where $q=\frac{p}{p-1}$. 
\end{example}

The construction of gradient curves through most points in the space was one of our original motivation for this project. It may be a bit surprising, however, that a measure on gradient curves arises in this setting. Besides naturally coming out of our duality argument, it also has a number of interesting connections to recent research:

\begin{itemize}
    \item In \cite{Cheeger1999}, blow-ups of Sobolev functions are shown to be generalized linear. This is used in \cite{Cheeger2016} to construct a measure on curves disintegrating the measure $\mu$ of the blown up space. Our result shows that such a disintegration of the measure exists without blowing up for $p$-harmonic functions.
    \item Measures on curve fragments, i.e. Alberti representations, have been an important tool in studying differential structures in metric measure spaces \cites{Bate2015, Schioppa2016-metric, Schioppa2016-derivations, Alberti2016}. Our result gives a natural construction of a measure on full curves, which is usually harder to obtain.
    \item Plans and test plans are measures on curves. They arise prominently in duality arguments for modulus \cite{Ambrosio2005} and in defining Sobolev spaces \cites{Gigli2015,Ambrosio2024preprint,AGS,Ambrosio2016}. Further, they can be used to \emph{represent the gradient} of a general Sobolev function, see \cites{Pasqualetto2022, Eriksson-Bique2024,Gigli2015}. The measure $\eta^*$ in our main theorem represents, in a sense, the gradient of a $p$-harmonic function.  It is also a dual measure to a generalized modulus problem, in analogy to \cite{Ambrosio2015}.
    \item Every finite measure on curves also defines a normal integral current $T$ in the sense of \cite{Ambrosio2000}, see \cite{paolini2012decomposition}*{Theorem 4.2.9}. This current depends on the orientation of the curves $\gamma$, and it makes sense to choose one such that $u\circ \gamma$ is increasing for $\eta^*$-a.e. $\gamma$. See also the discussion in \cites{Ambrosio2016,Dimarino, Ambrosio2024preprint} on the closely related notion of a derivation with a divergence. Indeed, there is a close connection between measures on curves, derivations and measurable vector fields \cites{Schioppa2016-derivations, Schioppa2016-metric}. The fact that our measure is supported on curves corresponds to the current being normal and the vector field having a measure valued divergence.
    \item Given $\eta^*$, one can consider the measure $\mu = e_{\ell_\gamma}^*(\eta^*)-e_0^*(\eta^*)$. Here we push-forward with the evaluation maps $e_{t}(\gamma)=\gamma(t)$, and one must first orient the curves $\gamma$ so that $u\circ \gamma$ is increasing. This measure should be thought of as the boundary of the integral current in the previous bullet point. It should also be connected closely with the Neumann data associated to $u$ see e.g. \cites{capogna2022neumann, Maly}. Further, this measure should also coincide with the $p$-harmonic measure considered in \cites{Bennewitz,HKM}. (Here, it should be noted that the term $p$-harmonic measure can refer to two completely different objects, cf. \cite{LLorente}.)
    \item  Gradient curves of Busemann functions are used in the classical proof of the Cheeger-Gromoll splitting theorem \cite{Cheeger1971} and its generalization to $RCD(K,N)$-spaces by Gigli \cite{Gigli2013splitting}. Roughly speaking, such gradient curves yield the splitting of the space into a product by identifying the curves parallel to the $\R$-factor. Indeed, one could apply our result to give a measure $\eta^*$ supported on the gradient curves of a Busemann function that represents the reference measure $\mu$ in the sense of (3) in Theorem \ref{thm:intro-main} (cf. \cite{Cheeger2016}). It may yield a different perspective to the proof of \cite{Gigli2013splitting}. In \cite{David2020preprint}, an analogous splitting theorem is shown for metric spaces that admit a similar measure on curves. 
\end{itemize}
We leave a more detailed exploration of these connections for future work. It seems that all of them express similar dual relationships that conjugate functions and forms encode in Euclidean spaces.

\subsection{Dual problem} In order to find the measure $\eta^*$ in Theorem \ref{thm:intro-main}, we first reformulate $p$-harmonicity in terms of a generalized modulus problem. 
The first observation is that if $u$ is $p$-harmonic in $\Om$, then its minimal $p$-weak upper gradient $g_u$ is the unique (weak) minimizer to a 
\begin{equation} \label{eq:intro-primal}
    \inf\Big\{
    \|\rho\|_{L^p(\overline{\Om})}:
    \rho\colon\overline{\Om}\to[0,\infty]
    \text{ Borel, and }
    A\rho\geq b\text{ in }\Gamma\Big\}.
\end{equation}
Here, $\Gamma\subset\Curv(\overline{\Om})$ is the set of all curves $\gamma$ that join two boundary points $x,y\in\partial\Om$, as defined in \eqref{eq:intro-Gamma}, \\ 
$b\colon\Gamma\to[0,\infty)$ is a function given by
$$ b(\gamma):=|u(x)-u(y)|
\quad\text{for all }\gamma\in\Gamma, $$
and if $\rho\colon\overline{\Om}\to[0,\infty]$ is a Borel function, then $A\rho\colon\Gamma\to[0,\infty]$ is a (Borel) function given by
$$ (A\rho)(\gamma):=\int_\gamma\rho ds
\quad\text{for all }\gamma\in\Gamma. $$
In this article we call \eqref{eq:intro-primal} a generalized modulus problem. Indeed, if $b\equiv 1$, then \eqref{eq:intro-primal} yields the classical modulus of the curve family $\Gamma$, $\Mod_p(\Gamma)$, see e.g. \cites{Heinonen2015,Ahlfors, Fuglede1957}. The equivalence of this generalized modulus problem and gradients of $p$-harmonic functions is established in Sections \ref{sec:generalization} and \ref{sec:energy}.

The reason we introduce \eqref{eq:intro-primal} is that it gives us a scalar optimization problem very similar to the one considered in \cite{Ambrosio2015}. There, the authors showed a general duality principle between modulus and (probability) measures on curves. Our key step is to apply  this approach (in a simplified form) to the problem \eqref{eq:intro-primal}.   
Subsequent to \cite{Ambrosio2015}, this duality argument has been used and refined in \cites{Fassler2019,David2020preprint,Durand-Cartagena2021,HonzlovaExnerova2018}. The concept of linear/convex duality is however much older than these and a topic of active research in the optimization community, see e.g. \cites{Barbu,Rockafellar}. We sketch here a slightly simplified version of our argument, see Sections \ref{sec:energy} and \ref{sec:dual} for details.

Let $\la h,\eta\ra:=\int_\Gamma hd\eta$ denote the (duality) pairing between a Borel function $h\colon\Gamma\to[0,\infty]$ and a Borel measure $\eta$ on the space of curves $\Gamma$. We define Lagrangian function
\begin{align*}
    \Lag(\rho,\eta):
    &=\|\rho\|_{L^p(\overline{\Om})}-\la A\rho-b,\eta\ra \\
    &=\|\rho\|_{L^p(\overline{\Om})}-\int_\Gamma(\int_\gamma\rho ds-b(\gamma))d\eta. 
\end{align*}
Then the problem in \eqref{eq:intro-primal} can be reformulated by using the Lagrangian:
\begin{align*}
    \inf_{\rho\in\A(\Gamma,b)}\|\rho\|_{L^p(\overline{\Om})}
    =\inf_{\rho\in\LB^p_+(\overline{\Om})}\sup_{\eta\in\M_+(\Gamma)}\Lag(\rho,\eta)
\end{align*}
where $\A(\Gamma,b):=\{\rho\colon\overline{\Om}\to[0,\infty]:\rho\text{ Borel, }A\rho\geq b\text{ in }\Gamma\}$, $\LB^p_+(\overline{\Om}):=\{\rho\colon\overline{\Om}\to[0,\infty]:\rho\text{ Borel, }\|\rho\|_{L^p(\overline{\Om})}<\infty\}$ and $\M_+(\Gamma)$ denotes the set of finite Radon measures on $\Gamma$.
The associated dual problem is
$$ \sup_{\eta\in\B(\Gamma)}\la b,\eta\ra =\sup_{\eta\in\M_+(\Gamma)}\inf_{\rho\in\LB^p_+(\overline{\Om})}\Lag(\rho,\eta) $$
where $\B(\Gamma):=\{\eta\in\M_+(\Gamma):\|\frac{d(A^\intercal\eta)}{d\mu}\|_{L^q(\overline{\Om})}\leq 1\}$. Here $A^\intercal$ denotes the formal transpose of $A$, as defined in Section \ref{ssec:A-and-its-transpose}, and $q=\frac{p}{p-1}$.
The rough idea is that both the primal problem and the dual problem attain minimum $\rho^*$ and maximum $\eta^*$, respectively, and in the presence of strong duality we have
\begin{equation} \label{eq:intro-strong-duality}
    \|\rho^*\|_{L^p(\overline{\Om})}=\Lag(\rho^*,\eta^*)
    =\int_\Gamma bd\eta^*.
\end{equation}
The maximizing measure $\eta^*$ is the measure obtained in Theorem \ref{thm:intro-main} and the properties 1-3 follow from \eqref{eq:intro-strong-duality}. 

A key technical issue is that the measure $\eta^*$ in Theorem \ref{thm:intro-main} need not be finite, and thus can not directly be obtained in this fashion. In fact, we prove exact duality in Proposition \ref{prop:strong-duality} for \eqref{eq:intro-primal} under a much more restricted set of assumptions than is applicable in Theorem \ref{thm:intro-main}. The measure $\eta^*$ is obtained by a further exhaustion and limiting argument, where strong duality is applied to a slightly regularized problem. This argument is executed in Section \ref{sec:proof-of-main}.

\subsection{Sheaf property}
As stated, the measure $\eta^*$ in Theorem \ref{thm:intro-main} is not necessarily unique. 
For example, if we equip Euclidean plane $\R^2$ with $\ell_1$-norm $|x|_1:=|x_1|+|x_2|$ and the Lebesgue measure $\lambda$, then linear function $f(x)=x_1+x_2$ has infinitely many different families of gradient curves. In this case, the minimal $p$-weak upper gradient is given by $g_f=|\nabla f|_\infty=\max(|\partial_x f|, |\partial_y f|)$ and the linear function $f$ is a $p$-harmonic function as a local minimizer of $u\mapsto \int |\nabla u|_\infty^p d\lambda$ for all $p\in (1,\infty)$.  This observation leads us to realize the failure of the so called sheaf property in general metric measure space.

The sheaf property  \cite{Bjorn2011} of $p$-harmonic functions says that if $\Om_1,\Om_2\subset X$ are domains and $u\in N^{1,p}(\Om_1\cup\Om_2)$ is $p$-harmonic in $\Om_1$ and $\Om_2$ separately, then it is $p$-harmonic in the union $\Om_1\cup\Om_2$. Such property holds trivially in the Euclidean setting, $X=\Rn$, because $u$ solves the $p$-Laplace equation \eqref{eq:intro-p-Laplace}
in the weak sense. Further, Gigli and Mondino \cite{Gigli2013}, showed that in metric spaces, whenever one can write a version of \eqref{eq:intro-p-Laplace}, then the sheaf property holds. Infinitesimally Hilbertian spaces are the most prominent examples of spaces where this applies. The sheaf property is so natural, that in axiomatic potential theory, which should be an extension of the Euclidean setting, the sheaf property is one of the axioms, see e.g. \cites{CC,Bauer} for the linear theory and \cite{HKM} for the non-linear theory. It is thus natural to conjecture that the sheaf property holds generally.

Further evidence is given by the fact that the energy $\int_\Omega g_u^pd\mu$ is built using a local gradient $g_u$. Being a minimizer for it would thus seem to suggest that the property is local.  However, in general, one can not derive any analogue -- even a weak one -- for the $p$-harmonic equation \eqref{eq:intro-p-Laplace}. This failure occurs already for minimizers of the energy  $\int |\nabla u|_\infty d\lambda$ in Euclidean spaces, since the norm $v\mapsto |v|_\infty$ is not differentiable. In fact, the methods of \cite{Gigli2013} rely heavily on the uniqueness of gradient vector fields which is related to the ability to make sense of \eqref{eq:intro-p-Laplace}. In contrast, we exploit the failure of this uniqueness. Indeed, our counter example is actually in the not so exotic metric measure space $X=\R^2$, equipped with the Lebesgue measure and the $\ell_1$-metric.

One of the reasons the sheaf problem has remained open for this long is that it is rather difficult to give examples of explicit $p$-harmonic functions in metric measure spaces. Further, it is hard to construct potential counter examples by just focusing on the minimization of the energy, since it is not clear how one might contradict minimality in $\Omega_1\cup \Omega_2$ while preserving minimality in both $\Omega_1$ and $\Omega_2$. Phrased in this way, it is difficult to even get started. Gradient curves, and our result on their existence, offer the key insight which leads to a simple counter example.  

Suppose that $u\in N^{1,p}(\Om_1\cup\Om_2)$ is $p$-harmonic in the domains $\Om_1$ and $\Om_2$ separately. By Theorem \ref{thm:intro-main} there exist families of gradient curves $\eta_1, \eta_2$ in $\Om_1$ and $\Om_2$. For $u$ to be $p$-harmonic in $\Om_1\cup \Om_2$, there would have to also be a family of gradient curves in $\Om_1\cup\Om_2$. However, due to the possible non-uniqueness of gradient curves, it may happen that the families $\eta_1, \eta_2$ do  not agree on the intersection $\Om_1\cap \Om_2$. Moreover, it is possible for them to be so ``fundamentally incompatible'' that $u$ must fail to be $p$-harmonic in the union $\Om_1\cup\Om_2$. For example, consider $\Om_1=(-1,1)\times(0,1), \Om_2=(0,1)\times(-1,1)$ which are two rectangles which intersect in a square $(0,1)\times (0,1)$. Suppose the support of $\eta_1$ consist of curves parallel to the $y$-axis and the curves in the support of $\eta_2$ consist of curves parallel to the $x$-axis. A moments thought shows that these two families can not be combined to a family in $\Omega_1\cup \Omega_2$ since $\eta_2$ has curves starting in $\partial \Omega_2 \cap \Om_1$, and $\eta_1$ has curves starting in $\partial \Omega_1 \cap \Om_2$. With this realization, the only task is then to cook up a piecewise defined function with these gradient curves. See Figure \ref{fig:regionsphi} for an illustration and Section \ref{sec:failuresheaf} of the formula. 

\begin{theorem}\label{thm:intro-failureofsheaf} There exists doubling metric measure space $X$ which satisfies a $p$-Poincar\'e inequality and domains $\Omega_1,\Omega_2\subset X$ and a function $u\in N^{1,p}(\Omega_1\cup\Omega_2)$ which is a $p$-harmonic in $\Omega_1$ and in $\Omega_2$, but is not a $p$-harmonic in $\Omega_1\cup \Omega_2$. 
\end{theorem}
We remark, that the proof that $u$ is not $p$-harmonic is quite simple and direct, and it does not go through Theorem \ref{thm:intro-main}. Gradient curves only play a role giving the right heuristic, which led us to the explicit formula for the counter example. This application demonstrates how gradient curves can help construct interesting examples of $p$-harmonic functions. It also shows how they contain useful geometric insights that may help tackle old problems.

\subsection{Structure of paper}

This article is organized as follows. In Section \ref{sec:preliminaries} we explain all the preliminaries, including the metric structure of the space of rectifiable curves. In Section \ref{sec:generalization} we define the general modulus problem \eqref{eq:intro-primal}. This generalized modulus problem is related to the minimal $p$-weak upper gradient of a $p$-harmonic function in Section \ref{sec:energy}. Duality of the generalized modulus problem is studied in Section \ref{sec:dual}, where a version of strong duality is shown. This, together with a delicate regularization argument, is used to prove Theorem \ref{thm:intro-main} in Section \ref{sec:proof-of-main}. The sheaf property is then discussed in Section \ref{sec:failuresheaf}, where we prove Theorem \ref{thm:intro-failureofsheaf}. The final section is independent of the previous sections and can be read independently. 
\section{Preliminaries} \label{sec:preliminaries}
Throughout the article, let $X=(X,d,\mu)$ be a metric measure space.
That is, $(X,d)$ is a metric space endowed with a complete Borel-measure $\mu$ such that $0<\mu(B)<\infty$ for all metric balls $B\subset X$. Throughout the article we assume that $X$ contains rectifiable curves.
\subsection{Notation for basic function spaces}
Let $Y$ be a metric space. We let $C(Y)$ denote the space of continuous functions and $LSC(Y)$ denote the space of lower semi-continuous functions. The space of Borel-measurable functions on $Y$ is denoted by $\Bor(Y)$. Subscript '$+$' stands for those functions that take only nonnegative values, for example 
$$ \Bor_+(Y):=\{f\in\Bor(Y):f\geq0\}. $$
For a function $f\colon Y\to[-\infty,\infty]$ we denote
$$ \|f\|_\infty:=\sup_{y\in Y}|f(y)|. $$

For a metric measure space $X$ and for any $1\leq p<\infty$, we let $L^p(X):=L^p(X,\mu)$ denote the usual Lebesgue space with norm 
$$ \|f\|_{L^p(X)}:=\Big(\int_X|f|^pd\mu\Big)^{1/p}, $$
and
$$ \LB^p(X):=\Big\{f\in\Bor(X):\int_X|f|^pd\mu<\infty\Big\}. $$
Subscript '$\loc$' stands for the local version of a space, for example $L^p_{\loc}(X)$ consists of measurable functions on $X$ that are locally $p$-integrable. If $f_i\in L^p(X)$ converge weakly to $f\in L^p(X)$ we write $f_i \xrightharpoonup{i\to\infty}f$.
\subsection{Rectifiable curves}
The space of rectifiable curves on $X$ is of great importance in this article. Let $a\leq b$ and consider a continuous curve $\gamma\colon[a,b]\to X$ that is rectifiable. By rectifiable, we mean that its length is finite:
$$ \ell(\gamma):=\sup\Big\{\sum_{i=0}^{n-1}d(\gamma(t_i),\gamma(t_{i+1})):n\in\N,a=t_0\leq\ldots\leq t_n=b\Big\}<\infty. $$
Sometimes it is appropriate to denote $\ell_\gamma:=\ell(\gamma)$. It follows directly from the definition that length of a curve is independent of the parametrization. More precisely, if $\sigma\colon [c,d]\to[a,b]$ is a nondecreasing continuous surjection, then $\ell(\gamma)=\ell(\gamma\circ\sigma)$.

By \cite{Ambrosio2004}*{Section 4.2} and \cite{Heinonen2001}*{Section 7.1} there exists a  unique arc-length reparametrization
$$ \tilde{\gamma}\colon[0,\ell(\gamma)]\to X $$
of $\gamma$, which is a continuous rectifiable curve such that $\ell(\tilde{\gamma}|_{[0,t]})=t$ for all $t\in [0,\ell(\gamma)]$. It holds that $\gamma=\tilde{\gamma}\circ V$, where $V(t):=\ell(\gamma|_{[a,t]})$ is called the length function of $\gamma$. The length function is continuous, nondecreasing surjection $[a,b]\to[0,\ell(\gamma)]$. 

If $\rho\in\Bor_+(X)$ or $\rho\in C(X)$, the curve integral of $\rho$ along $\gamma$ is defined through arc-length reparametrization $\tilde{\gamma}$ as
$$ \int_{\gamma}\rho ds:=\int_0^{\ell(\gamma)}\rho(\tilde{\gamma}(t))dt. $$
We define the space of rectifiable curves on $X$ by
$$ \Curv(X):=\{\gamma\colon[0,\ell(\gamma)]\to X:\gamma\text{ is continuous, rectifiable and arc-length parametrized}\}. $$
This definition is the same as in \cite{Bjorn2011}, except we include constant curves into $\Curv(X)$.
\subsection{Newtonian space, gradient curves and \texorpdfstring{$p$}{p}-harmonic functions}
The Newtonian space is defined as
$$ N^{1,p}(X):=\{f\colon X\to[-\infty,\infty]:\|f\|_p+\inf_g\|g\|_p<\infty\} $$
where the infimum is taken over all upper gradients $g$ of $f$. A function $g\in\Bor_+(X)$ is an upper gradient of $f$ if
\begin{equation} \label{eq:upper-gradient}
    |f(x)-f(y)|\leq \int_\gamma gds
\end{equation}
for all nonconstant $\gamma\in\Curv(X)$ with endpoints $x$ and $y$.  We say that $g$ is an upper gradient for $f$ along $\gamma$, if \eqref{eq:upper-gradient} holds for all subcurves of $\gamma$. Here, we use the convention $|\infty-\infty|=\infty$. If there exists a collection of curves $\Gamma\subset\Curv(X)$ such that $\Mod_p(\Gamma)=0$ and \eqref{eq:upper-gradient} holds for all nonconstant curves except those in $\Gamma$, then we say that $g$ is a $p$-weak upper gradient of $f$. Here 
$$ \Mod_p(\Gamma):=\inf\Big\{\int_X\rho d\mu:\int_\gamma\rho ds\geq 1\text{ for all }\gamma\in\Gamma\Big\} $$
is the $p$-modulus of $\Gamma$. If a property on the space of curves $\Curv(X)$ holds up to a set of zero $p$-modulus, we say that is holds  $p$-almost everywhere, $p$-a.e. for short. If $\Gamma\subset\Curv(X)$ is such that $\Mod_p(\Gamma)=0$, we say that $\Gamma$ is $p$-exceptional. 

If $\Om\subset X$ is open, then the Newtonian space on $\Om$ with zero boundary value can be defined as
\begin{equation} \label{eq:N1p-with-zero-boundary-value}
    N^{1,p}_0(\Om):=\{f|_{\Om}:f\in N^{1,p}(\overline{\Om})\text{ and }f=0\text{ on }\partial\Om\},
\end{equation}
see \cite{Bjorn2011}*{Section 2.7}.

We will make frequent use of the following Fuglede's Lemma, \cite{Fuglede1957}*{Theorem 3}.

\begin{lemma}[Fuglede's lemma] \label{lem:Fuglede}
    If $(f_i)$ is a sequence of functions in $\LB^p(X)$ and $f\in\LB^p(X)$ such that $f_i\xrightarrow{i\to\infty}f$ in $L^p(X)$, then there exists a subsequence $(f_{i_k})$ such that
    \[
    \lim_{k\to\infty}\int_\gamma f_{i_k} ds = \int_\gamma f ds, 
    \]
    for $p$-a.e. $\gamma\in\Curv(X)$.
\end{lemma}
Indeed, a simple consequence of this lemma is that the (convex) collection of $p$-weak upper gradients is closed under $L^p(X)$ convergence. If $f\in N^{1,p}(X)$, there exists a minimal $p$-weak upper gradient $g_f\in \LB^p_+(X)$ which is unique up to a set of measure zero, see \cite{Bjorn2011}*{Theorem 2.5}. By minimal we mean that $g_f\leq g$ $\mu$-a.e. for all $p$-weak upper gradients $g$ of $f$.

\begin{definition} \label{def:gradient-curve}
Let $f\in N^{1,p}(X)$. We say that a curve $\gamma\in\Curv(X)$ is a gradient curve of $f$ if $g_f$ is an upper gradient for $f$ along $\gamma$ and 
\begin{equation} \label{eq:gradient-curve-def}
    |f(\gamma(0))-f(\gamma(\ell_\gamma))|=\int_{\gamma}g_fds.
\end{equation}
\end{definition}

\begin{remark}\label{rmk:gradientcurve}
If $\gamma\in\Curv(X)$ is a gradient curve of $f\in N^{1,p}(X)$, then it follows that $f\circ\gamma$ is absolutely continuous and
$$ |f(\gamma(a))-f(\gamma(b))|=\int_{\gamma|_{[a,b]}}g_fds \quad\text{for all }0\leq a\leq b\leq \ell_\gamma. $$
Indeed, it follows from \eqref{eq:gradient-curve-def} that
\begin{align*}
    \int_{\gamma|_{[0,a]}}g_fds+\int_{\gamma|_{[a,b]}}g_fds
    +\int_{\gamma|_{[b,\ell_\gamma]}}g_fds
    \leq
    |f(\gamma(0))-f(\gamma(a))|
    +|f(\gamma(a))-f(\gamma(b))|
    +|f(\gamma(b)-f(\gamma(\ell_\gamma))|.
\end{align*}
Since $g_f$ is an upper gradient for $f$ along $\gamma$, we may estimate the left hand side from below by applying \eqref{eq:upper-gradient} to the sub-curves $\gamma|_{[0,a]}$ and $\gamma|_{[b,\ell_\gamma]}$ to obtain
\[
\int_{\gamma|_{[a,b]}}g_fds\leq |f(\gamma(a))-f(\gamma(b))| 
\]
from which the claim follows.
\end{remark}

\begin{definition}[\cite{Bjorn2011}*{Definition 7.7}] \label{def:p-harmonic}
Let $\Om\subset X$ be open and nonempty. A function $u\in N^{1,p}_{\loc}(\Om)$ is $p$-harmonic if it is continuous and if 
$$\int_{\varphi\neq0}g_u^pd\mu\leq\int_{\varphi\neq 0}g_{u+\varphi}^pd\mu $$
for all compactly supported Lipschitz continuous functions $\varphi\colon\Om\to\R$.
\end{definition}

Here, we assume that $u$ is also continuous, since otherwise $u$ may be modified on a zero capacity subset. With this definition, we avoid the need to discuss the choice of representatives. In general, for doubling spaces satisfying a $p$-Poincar\'e inequality, every $p$-harmonic function (without the continuity assumed a priori) has a  continuous representative; see e.g. \cite{Bjorn2011}*{Section 8}.
\subsection{Space of rectifiable curves as a metric space}
For the purpose of this article we need to topologize the space of rectifiable curves $\Curv(X)$. We heavily employ the theory presented in \cite{Savare2022}*{Section 2.2}.

Let $\gamma\in\Curv(X)$ and define its unit-interval reparametrization
$R_\gamma\colon[0,1]\to X$ by
\[ R_\gamma(t):=\gamma(\ell(\gamma)t) \quad\text{for all }t\in[0,1]. \]
Note that $R_\gamma\in C([0,1];X)$. We equip $C([0,1];X)$ with the standard sup-metric
$$ d_C(R_1,R_2):=\sup_{t\in[0,1]}d(R_1(t),R_2(t)) \quad\text{for all }
R_1,R_2\in C([0,1];X). $$
Following \cite{Savare2022}*{Section 2.2.2} we define an equivalence relation on $C([0,1];X)$ as follows. Let 
$$ \Sigma:=\{\sigma\colon[0,1]\to[0,1]:
\sigma\text{ continuous, nondecreasing surjection}\}. $$
For $R_1,R_2\in C([0,1];X)$ we write $R_1\sim R_2$ if there exist $\sigma_1,\sigma_2\in\Sigma$ such that
$$ R_1\circ\sigma_1=R_2\circ\sigma_2. $$
By \cite{Savare2022}*{Corollary 2.2.5} relation $\sim$ is indeed an equivalence relation. We denote the equivalence class of $R\in C([0,1];X)$ by $[R]$. By \cite{Savare2022}*{Proposition 2.2.6}, the quotient space of arcs $A(X):=C([0,1];X)/\sim$ is a metric space with metric 
$$ d_A([R_1],[R_2]):=\inf_{\sigma_i\in\Sigma}d_C(R_1\circ\sigma_1,R_2\circ\sigma_2). $$
Moreover, the quotient topology of $A(X)$ coincides with the metric topology induced by $d_A$.

\begin{lemma} \label{lem:equivalence-classes-have-same-arc-length-reparam}
Rectifiable curves $R_1,R_2\in C([0,1];X)$ have the same arc-length reparametrization if and only if $R_1\sim R_2$.
\end{lemma}
\begin{proof}
"$\Longrightarrow$": Let $\gamma$ denote the arc-length reparametrization of $R_1$ and $R_2$. We may assume that $\ell:=\ell(\gamma)>0$, because otherwise $\gamma$ is a constant curve, and consequently $R_1=R_2$.
We prove that $R_1\sim R_\gamma$, and an analogous argument shows that $R_2\sim R_\gamma$. Then $R_1\sim R_2$.

Since $\gamma$ is the arc-length reparametrization of $R_1$, we have $R_1=\gamma\circ V$, where $V\colon [0,1]\to[0,\ell]$ is the length function of $R_1$.
On the other hand, $R_\gamma=\gamma\circ s$, where $s\colon[0,1]\to[0,\ell]$ is the scaling $s(t)=\ell t$. Then
$$ R_1=R_\gamma\circ s^{-1}\circ V. $$
Since $s^{-1}\circ V\in\Sigma$, we conclude that $R_1\sim R_\gamma$.

"$\Longleftarrow$": Suppose that $R_1\sim R_2$. Then $R_1\circ\sigma_1=R_2\circ\sigma_2$ for some $\sigma_1,\sigma_2\in\Sigma$. 
It suffices to show that $R_1\circ\sigma_1$ and $R:=R_2\circ\sigma_2$ have the same arc length reparametrization.

Let $\gamma_1$ be the arc-length reparametrization of $R_1$ and $\gamma$ be that of $R$. Then
$$ R_1=\gamma_1\circ V_1\quad\text{and}\quad R=\gamma\circ V $$
where $V_1\colon[0,1]\to[0,\ell]$ and $V\colon[0,1]\to[0,\ell]$ are the length functions of $R_1$ and $R$, respectively and $\ell:=\ell(R_1)=\ell(R)$.
Thus
$$ \gamma_1\circ V_1\circ\sigma_1=\gamma\circ V. $$
We check that $V=V_1\circ\sigma_1$, and thus $\gamma_1=\gamma$. Indeed, for any $t\in[0,1]$
\begin{align*}
    V(t)
    &=\ell((R_1\circ\sigma_1)|_{[0,t]}) \\
    &=\sup\Big\{\sum_{i=0}^{n-1}d(R_1(\sigma_1(t_i)),R_1(\sigma_1(t_{i+1}))):n\in\N,0=t_0\leq\ldots\leq t_n=t\Big\} \\
    &=\sup\Big\{\sum_{i=0}^{n-1}d(R_1(s_i),R_1(s_{i+1})):n\in\N, 0=s_0\leq \ldots\leq s_n=\sigma_1(t)\Big\} \\
    &=V_1(\sigma_1(t)).
\end{align*}
Above we used the fact that $\sigma_1$ is a continuous nondecreasing surjection $[0,1]\to[0,1]$.
\end{proof}

With the aid of Lemma \ref{lem:equivalence-classes-have-same-arc-length-reparam} we can see that the mapping $\gamma\mapsto [R_\gamma]$ yields an injection $\Curv(X)\to A(X)$ and we can define a metric on $\Curv(X)$ by
\begin{equation} \label{eq:d-Curv-def}
    d_{\Curv}(\gamma_1,\gamma_2):=
    d_A([R_{\gamma_1}],[R_{\gamma_2}])
    \quad\text{for all }\gamma_1,\gamma_2\in\Curv(X).
\end{equation}
\subsection{Properties of \texorpdfstring{$(\Curv(X),d_{\Curv})$}{CurvX}}
Here, we list certain continuity properties of maps on $(\Curv(X),d_{\Curv})$ and compactness properties of space $(\Curv(X),d_{\Curv})$.

\begin{lemma} \label{lem:regularity-of-various-maps-on-Curv}
We have the following regularity properties of various functions on $\Curv(X)$:
\begin{enumerate}
    \item Function $\ell\colon\Curv(X)\to[0,\infty)$ is lower semi-continuous.
    \item If $\rho\colon X\to[0,\infty]$ is lower semi-continuous then the function $\gamma\mapsto \int_\gamma\rho ds$ is lower semi-continuous on $\Curv(X)$. 
    \item If $\rho\in\Bor_+(X)$, then the function $\gamma\mapsto \int_\gamma\rho ds$ is Borel on $\Curv(X)$.
    \item Endpoint evaluation maps
    \begin{align*}
        \begin{cases}
        e_0\colon\Curv(X)\to X,&\quad e_0(\gamma):=\gamma(0), \\
        e_\ell\colon\Curv(X)\to X,&\quad e_\ell(\gamma):=\gamma(\ell_\gamma),
        \end{cases}
    \end{align*}
    are continuous.
\end{enumerate}
\end{lemma}

\begin{proof}
We employ the analogous results of \cite{Savare2022} for functions defined on $(A(X),d_A)$. The definition of $d_{\Curv}$ via $d_A$ in \eqref{eq:d-Curv-def} allows us to state these results for functions defined on $(\Curv(X),d_{\Curv})$.
\begin{enumerate}
    \item Follows from \cite{Savare2022}*{Lemma 2.2.11(d)}.
    \item Follows from \cite{Savare2022}*{Lemma 2.2.11(e)}.
    \item Follows from \cite{Savare2022}*{Theorem 2.2.13(e)}.
    \item Follows from \cite{Savare2022}*{Lemma 2.2.11(d)}.
\end{enumerate}
\end{proof}

\begin{proposition} \label{prop:compactness}
Suppose that $X$ is compact and $\Gamma\subset\Curv(X)$ is a family of curves whose lengths are uniformly bounded from above. Then $\Gamma$ is precompact.
\end{proposition}

\begin{proof}
Let $(\gamma_n)\subset\Gamma$ be a sequence of curves. Consider the sequence of their unit-length reparametrizations, $(R_{\gamma_n})\subset C([0,1];X)$. Each $R_{\gamma_n}$ is $\ell(\gamma_n)$-Lipschitz, and $\ell(\gamma_n)\leq C$ for all $n\in\N$ and for some constant $C>0$.
Since $X$ is compact, we may apply Arzel\`a-Ascoli theorem to find a subsequence of $(R_{\gamma_n})$, still denoted by itself, and a curve $R\in C([0,1];X)$ such that 
$$ d_C(R_{\gamma_n},R)\xrightarrow{n\to\infty}0. $$
Moreover, $R$ is rectifiable, because
$$ \ell(R)\leq \liminf_{n\to\infty}\ell(\gamma_n)\leq C<\infty. $$
We conclude that $R$ has a unique arc-length reparametrization $\gamma\colon [0,\ell(R)]\to X$. We claim that $\gamma_n\xrightarrow{n\to\infty}\gamma$ in $\Curv(X)$. Indeed,
\begin{align*}
    d_{\Curv}(\gamma_n,\gamma)
    =d_A([R_{\gamma_n}],[R_\gamma])
    =d_A([R_{\gamma_n}],[R])
    \leq d_C(R_{\gamma_n},R)\xrightarrow{n\to\infty}0.
\end{align*}
Above we employed the fact that $R_\gamma$ and $R$ have the same arc-length reparametrization, and thus $R_\gamma\sim R$ by Lemma \ref{lem:equivalence-classes-have-same-arc-length-reparam}.
\end{proof}

\begin{lemma} \label{lem:curves-with-bdd-length-compact}
If $X$ is compact and $C>0$ is a constant, then $\{\gamma\in\Curv(X):\ell(\gamma)\leq C\}$ is compact.
\end{lemma}

\begin{proof}
The set $\{\gamma\in\Curv(X):\ell(\gamma)\leq C\}$ is closed by Lemma \ref{lem:regularity-of-various-maps-on-Curv} and precompact by Proposition \ref{prop:compactness}. Hence it is compact. 
\end{proof}

\begin{corollary}
If $X$ is compact, then $\Curv(X)$ is $\sigma$-compact.
\end{corollary}

Given a curve $\gamma\colon[a,b]\to X$, we define the diameter of its image as
$$ \operatorname{diam}(\gamma):=\operatorname{diam}(\gamma([a,b]))=\sup\{d(\gamma(t),\gamma(s)):t,s\in[a,b]\}. $$
Similarly to length, it is clear that diameter is independent of the reparametrization of a curve. In particular, if $R_1,R_2\in C([0,1];X)$ and $R_1\sim R_2$, then $\operatorname{diam}(R_1)=\operatorname{diam}(R_2)$.

\begin{lemma} \label{lem:diam-continuous}
Diameter function $\operatorname{diam}\colon\Curv(X)\to[0,\infty)$ is continuous.
\end{lemma}

\begin{proof}
We begin by proving that $\operatorname{diam}\colon C([0,1];X)\to \R$ is continuous with respect to $d_C$. First note that it is lower semi-continuous as a supremum over continuous functions $\gamma\mapsto d(\gamma(t),\gamma(s))$. For the upper semi-continuity, fix $R\in C([0,1];X)$ and let $(R_n)$ be a sequence of curves such that $R_n\to R$ in $C([0,1];X)$, i.e. $d_C(R_n,R)\xrightarrow{n\to\infty}0$. 
Given $\epsilon>0$, there exists $N\in \N$ such that $n\geq N$ implies that 
$$ R_n([0,1])\subset B(R([0,1]),\epsilon). $$
Here we denote $B(A,\epsilon):=\{x\in X:d(x,A)<\epsilon\}$ for any $A\subset X$. Note that if $A\subset X$, then
$$ \operatorname{diam}(B(A,\epsilon))\leq \operatorname{diam}(A)+2\epsilon. $$
We conclude that 
$$ \operatorname{diam}(R_n)\leq\operatorname{diam}(B(R([0,1]),\epsilon))\leq \operatorname{diam}(R)+2\epsilon. $$
It follows that
$$ \limsup_{n\to\infty}\operatorname{diam}(R_n)\leq\operatorname{diam}(R)+2\epsilon. $$
for any $\epsilon>0$ and thus $\operatorname{diam}$ is upper semi-continuous with respect to $d_C$.

Next, we show that $\operatorname{diam}\colon A(X)\to \R$ is continuous with respect to $d_A$. Note that $\operatorname{diam}$ is well-defined on $A(X)$, because it is independent of the equivalence class. Since the quotient topology of $A(X)$ agrees with the metric topology induced by $d_A$, we conclude that $\operatorname{diam}\colon A(X)\to \R$ is continuous. Finally, the continuity of $\operatorname{diam}\colon\Curv(X)\to\R$ follows from the continuity of $\operatorname{diam}\colon A(X)\to\R$ and the definition of $d_{\Curv}$. 
\end{proof}
\subsection{Borel measures and the space of Radon measures}

Let $Y$ be a metric space. We point out that in this paper either $Y=X$ or $Y\subset\Curv(X)$. Let $\eta$ be a (nonnegative complete) Borel measure on $Y$. If $A\subset Y$ is Borel, the restriction of $\eta$ into $A$ is given by
\[ (\eta\resmes A)(E):=\eta(A\cap E) \quad\text{for all }E\subset Y\text{ Borel (or measurable).} \]
The restricted measure $\eta\resmes A$ is a Borel measure on $A$.
Conversely, if $A\subset Y$ is Borel and $\eta$ is a Borel measure defined on $A$, we may extend $\eta$ into the whole of $Y$ by zero extension,
$$ \eta'(E):=\eta(A\cap E)\quad\text{for all }E\subset \text{ Borel (or measurable),} $$
and $\eta'$ is a Borel measure on $Y$. Throughout this article, whenever a Borel measure $\eta$ is defined on a subset $A\subset Y$ we write simply
$$ \int_Yhd\eta:=\int_Yhd\eta'=\int_Ahd\eta
\quad\text{for all }h\in\Bor_+(Y) $$
by abuse of notation. In particular $\int_Yhd(\eta\resmes A)=\int_Ahd\eta$.

The space of finite, signed Radon measures on $Y$ is denoted by $\M(Y)$ and $\M_+(Y)$ stands for finite, nonnegative Radon measures on $Y$. Note that all Radon measures are also Borel. The weak topology on $\M(Y)$ is defined to be the coarsest topology such that 
$$ \eta\mapsto\int_Yhd\eta $$
is continuous for all bounded $h\in C(Y)$.
We say that a sequence $(\eta_n)$ of Radon measures in $\M_+(Y)$ converges weakly to $\eta\in\M_+(Y)$ if
$$ \int_Yhd\eta_n\xrightarrow{n\to\infty}\int_Yhd\eta\quad\text{for all bounded }h\in C(Y). $$

Prokhorov's theorem \cite{Bogachev}*{Theorem 8.6.7} is important for us. A collection $\mathcal{K}\subset\M_+(Y)$ is uniformly bounded if $\sup_{\eta \in \mathcal{K}}\eta(Y)<\infty$ and it is called equi-tight, if for every $\epsilon>0$ there exists a compact set $K_\epsilon \subset Y$ for which $\eta(Y\setminus K_\epsilon)<\epsilon$ for all $\eta \in \mathcal{K}$.

\begin{theorem}[Prokhorov] \label{thm:prokhorov} Let $Y$ be a complete separable metric space. 
Let $\mathcal{K}\subset\M_+(Y)$ be a collection of Radon measures that is uniformly bounded and equi-tight. Then $\mathcal{K}$ is  pre-compact in $\M_+(Y)$ with respect to the weak topology.
\end{theorem}

If $Y$ is a compact metric space, then the dual of the space of continuous functions on $Y$ is the space of (signed) Radon measures on $Y$ by Riesz representation theorem. We denote this dual pairing by
$$ \la h,\eta\ra:=\int_Yhd\eta $$
for all $h\in C(Y)$ and $\eta\in\M(Y)$. We extend this notation to cover that case where $h\in\Bor(Y)$ and $\eta$ is a Borel measure on $Y$ such that $\la h,\eta\ra:=\int_Yhd\eta\in[-\infty,\infty]$ is well defined.
\subsection{Path integration operator \texorpdfstring{$A$}{A} and its formal transpose \texorpdfstring{$A^\intercal$}{AT}} \label{ssec:A-and-its-transpose}
In this section we introduce some shorthand notation that is used throughout the article. The motivation for this notation comes from basic convex optimization, see e.g. \cites{Barbu, Rockafellar}.

We define the path-integration operator $A$ on $\Bor_+(X)$ as follows: If $\rho\in\Bor_+(X)$, then $A\rho$ is a function on $\Curv(X)$, given by
$$ A\rho\colon \gamma\mapsto \int_\gamma\rho ds \quad\text{for all }\gamma\in\Curv(X). $$
Lemma \ref{lem:regularity-of-various-maps-on-Curv} implies that $A\colon\Bor_+(X)\to\Bor_+(\Curv(X))$ and  $A\colon LSC_+(X)\to LSC_+(\Curv(X))$. 

If $\eta$ is a Borel measure on a Borel set $\Gamma\subset\Curv(X)$, we define Borel measure $A^\intercal\eta$ on $X$ by
$$ (A^\intercal\eta)(E):=\int_\Gamma\int_\gamma\mathbbm{1}_Edsd\eta\quad\text{for all }E\subset X\text{ Borel.} $$ 
Here and subsequently $\mathbbm{1}_E$ stands for indicator function of set $E$.
It follows that if $\rho\in\Bor_+(X)$, then 
\begin{equation} \label{eq:transpose-formula-new}
    \la A\rho,\eta\ra=\int_\Gamma\int_\gamma \rho dsd\eta=\int_X\rho d(A^\intercal\eta)=\la \rho,A^\intercal\eta\ra,
\end{equation}
hence we may view $A^\intercal$ as a formal transpose of $A$. Measure $A^\intercal\eta$ is also called "induced measure" in \cite{David2020preprint}*{Definition 4.4} or "projected measure" in \cite{Savare2022}*{Section 4.2}.

\begin{remark}
Both $A$ and $A^\intercal$ are linear and nonnegative. 
\end{remark}

\begin{remark} \label{rem:restriction-and-AT}
If $\eta$ is a Borel measure on a Borel set $\Gamma\subset\Curv(X)$ and $\tilde{\Gamma}\subset\Gamma$ is Borel, then $A^\intercal(\eta\resmes\tilde{\Gamma})\leq A^\intercal\eta$.
\end{remark}
\subsection{Lower semi-continuity results for \texorpdfstring{$A$}{A} and \texorpdfstring{$A^\intercal$}{AT}}
The following lemma says that for fixed $\rho\in C_+(X)$ the function $\eta\mapsto\la A\rho,\eta\ra=\la \rho,A^\intercal\eta\ra$ is lower semi-continuous with respect to the weak convergence of Radon measures. For technical reasons, we need a slightly weaker notion of convergence than weak convergence in the statement. 
 
\begin{lemma} \label{lem:lsc-of-dual-pairing}
Let $\Gamma\subset\Curv(X)$ be Borel. If $(\eta_n)$ is a sequence in $\M_+(\Gamma)$ and $\eta$ is a Borel measure on $\Gamma$ for which $\lim_{n\to \infty}\int_{\Gamma} h d\eta_n = \int_{\Gamma} h d\eta$ for all $h\in C_+(\Gamma)$, then for all $\rho\in C_+(X)$
\begin{enumerate}
    \item $\la A\rho,\eta\ra\leq\liminf_{n\to\infty}\la A\rho,\eta_n\ra$,
    \item $\la \rho,A^\intercal\eta\ra\leq\liminf_{n\to\infty}\la \rho,A^\intercal\eta_n\ra$. 
\end{enumerate}
\end{lemma}

\begin{proof}
\begin{enumerate}
    \item Let $\rho\in C_+(X)$. Then $A\rho\in LSC_+(\Gamma)$ by Lemma \ref{lem:regularity-of-various-maps-on-Curv}, and consequently we can find a sequence $(h_k)$ of nonnegative, bounded, continuous functions on $\Gamma$ such that $h_k\nearrow A\rho$. Then for all $k\in\N$ 
    $$ \lim_{n\to\infty}\int_\Gamma h_k d\eta_n=\int_\Gamma h_k d\eta $$
    by the weak convergence of $(\eta_n)$.
    By the monotonicity of $(h_k)$ we may let $k\to\infty$ to obtain
    $$ \int_\Gamma A\rho d\eta=\lim_{k\to\infty}\int_\Gamma h_k d\eta=\lim_{k\to\infty}\lim_{n\to\infty}\int_\Gamma h_kd\eta_n\leq \liminf_{n\to\infty}\int_\Gamma A\rho d\eta_n. $$
    It follows that for all $\rho\in C_+(X)$
    $$ \la\rho,A^\intercal\eta\ra \leq \liminf_{n\to\infty}\,\la \rho,A^\intercal\eta_n\ra. $$
    \item This follows from item (1) and \eqref{eq:transpose-formula-new}.
\end{enumerate}
\end{proof}

The following lemma is a corollary of the previous lemma, Lemma \ref{lem:lsc-of-dual-pairing}. It yields a lower semi-continuity result for the joint function $(\rho,\eta)\mapsto\la A\rho,\eta\ra=\la \rho,A^\intercal\eta\ra$ under suitable assumptions. 

\begin{lemma} \label{lem:B-weakly-closed-and-double-limit-in-dual-pairing}
Let $\Gamma\subset\Curv(X)$ be Borel. Suppose that $(\eta_n)$ is a sequence in $\M_+(\Gamma)$ and $\eta$ is a Borel measure on $\Gamma$ for which $\lim_{n\to \infty}\int_{\Gamma} h d\eta_n = \int_{\Gamma} h d\eta$ for all $h\in C_+(\Gamma)$, and for which $\|\frac{d(A^\intercal\eta_n)}{d\mu}\|_{L^q(X)}\leq 1$.
Then the following hold:
\begin{enumerate}
    \item Measure $A^\intercal \eta$ is absolutely continuous with respect to $\mu$ and $\|\frac{d(A^\intercal\eta)}{d\mu}\|_{L^q(X)}\leq 1$.
    \item If $(\rho_n)$ is a sequence in $\LB^p_+(X)$ such that $\rho_n\xrightarrow{n\to\infty}\rho\in\LB^p_+(X)$ in $L^p(X)$, then
    \[
    \langle A\rho,\eta\rangle \leq \liminf_{n\to \infty} \langle A\rho_n,\eta_n\rangle
    \quad\text{and}\quad
    \langle \rho, A^\intercal \eta\rangle \leq \liminf_{n\to \infty} \langle \rho_n, A^\intercal \eta_n\rangle.
    \]
\end{enumerate}
\end{lemma}

\begin{proof}
\begin{enumerate}
\item By Lemma \ref{lem:lsc-of-dual-pairing} for all $\rho\in C_+(X)$
$$ \int_X\rho d(A^\intercal\eta)\leq \liminf_{n\to\infty}\int_X\rho d(A^\intercal\eta_n)
=\liminf_{n\to\infty}\int_X\rho\frac{d(A^\intercal\eta_n)}{d\mu}d\mu=\int_X\rho fd\mu, $$
where $f\in L^q(X)$ denotes a weak limit of the sequence $\big(\frac{d(A^\intercal\eta_n)}{d\mu}\big)_{n=1}^\infty$. 

Let $\nu:=f\mu$. We conclude that $(A^\intercal\eta)(E)\leq \nu(E)$ for all $E\subset X$ Borel, and thus $A^\intercal\eta$ is absolutely continuous with respect to $\mu$. It follows that $\|\frac{d(A^\intercal\eta)}{d\mu}\|_{L^q(X)}\leq\|f\|_{L^q(X)}\leq 1$.
\item Fix $\epsilon>0$, and choose a continuous function $\tilde{\rho}\in L^p(X)\cap C_+(X)$ such that $\|\rho-\tilde{\rho}\|_{L^p(X)}\leq \epsilon$. Then
\begin{align*}
    \la\rho,A^\intercal\eta\ra &=\la \rho, A^\intercal \eta \ra \\&\leq \la \rho-\tilde{\rho}, A^\intercal \eta \ra + \la \tilde{\rho}, A^\intercal \eta \ra \\
    &\leq \epsilon + \la \tilde{\rho}, A^\intercal \eta \ra \\
    &\leq \epsilon + \liminf_{n\to\infty}\la \tilde{\rho}, A^\intercal \eta_n \ra & \text{Follows from Lemma \ref{lem:lsc-of-dual-pairing}.} \\
    &=\epsilon + \liminf_{n\to\infty}\Big(\la \rho_n, A^\intercal \eta_n \ra  + \la \tilde{\rho}-\rho_n, A^\intercal \eta_n \ra\Big)\\
    &\leq 2\epsilon + \liminf_{n\to\infty}\la \rho_n, A^\intercal \eta_n \ra \\
    &=2\epsilon + \liminf_{n\to\infty}\la A\rho_n,\eta_n \ra.
\end{align*}
Above we employed identity \eqref{eq:transpose-formula-new} in the first and the last equality.
Since $\epsilon>0$ is arbitrary the claim follows.
\end{enumerate}
\end{proof}
\section{A Generalization of modulus problem} \label{sec:generalization}

Let $X=(X,d,\mu)$ be a metric measure space. 

Fix $1<p<\infty$ and let $q=\frac{p}{p-1}$ denote the dual exponent of $p$. We introduce a generalized modulus problem:

\begin{definition}
Let us fix a family of rectifiable curves $\Gamma\subset\Curv(X)$ and a (bounding) function $b\colon\Gamma\to[-\infty,\infty]$. 
\begin{itemize}
    \item[(i)] A function $\rho\in\Bor_+(X)$ is admissible for $\Gamma$ with respect to the bounding function $b$ if $\int_\gamma\rho ds\geq b(\gamma)$ for all $\gamma\in\Gamma$. The class of all such admissible functions is denoted by $\A(\Gamma,b)$.
    \item[(ii)] The $p$-modulus of $\Gamma$ with respect to $b$ is
    $$ \Mod_p(\Gamma,b):=\inf\Big\{\int_X\rho^pd\mu:\rho\in\A(\Gamma,b)\Big\}. $$
    \item[(iii)] The continuous $p$-modulus of $\Gamma$ with respect to $b$ is
    $$ \Mod_{p}^c(\Gamma,b):=\inf\Big\{\int_X\rho^pd\mu:\rho\in\A(\Gamma,b)\cap C(X)\Big\}. $$
\end{itemize}
\end{definition}

\begin{remark}
If $b\equiv1$, then $\Mod_p(\Gamma,1)$ and $\Mod_p^c(\Gamma,1)$ reduce back to the usual (continuous) modulus of a family of curves, $\Mod_p(\Gamma)$ and $\Mod_p^c(\Gamma)$; cf. \cites{Heinonen2001,Heinonen2015,Bjorn2011}.
\end{remark}

The following lemmas give the existence of a weak minimizer of the (continuous) generalized modulus problem. A function $\rho\in\Bor_+(X)$ is called weakly admissible if 
$\int_\gamma\rho ds\geq b(\gamma)$ for $p$-a.e. $\gamma\in\Gamma$. Weakly admissible functions are characterized by Fuglede's lemma.

\begin{lemma}\label{lem:closure} Let $\Gamma \subset \Curv(X)$ and $b:\Gamma\to [-\infty,\infty]$. Then
 $\rho \in \overline{\mathcal{A}(\Gamma,b)}\cap \mathfrak{L}_+^p(X)$ if and only if $\rho \in \mathfrak{L}_+^p(X)$ and it is weakly admissible, i.e. $\int_\gamma \rho ds \geq b(\gamma)$ for $p$-a.e. $\gamma\in \Gamma$.
\end{lemma}
\begin{proof}
    If $\rho \in \overline{\mathcal{A}(\Gamma,b)}\cap \mathfrak{L}_+^p(X)$, then there exists a sequence of $\rho_k\in \mathcal{A}(\Gamma,b)$ converging to $\rho$ in $L^p(X)$. Then, the weak admissibility of $\rho$ follows from Fuglede's Lemma  \ref{lem:Fuglede} after passing through a subsequence.

    Conversely, suppose that $\rho \in \mathfrak{L}_+^p(X)$ is weakly admissible. Then, there exists a collection $E\subset \Gamma$ such that $\int_\gamma \rho ds \geq b(\gamma)$ for all $\gamma\in \Gamma \setminus E$ and $\Mod_p(\Gamma)=0$. By \cite{Heinonen2015}*{Lemma 5.2.8} there exists a function $h\in \mathfrak{L}^p_+(X)$ such that $\int_\gamma h ds = \infty$ for all $\gamma \in E$. Let $\rho_k = \rho + h/k$. Then $\rho_k \in \mathcal{A}(\Gamma,b)$ and clearly $\rho_k\to \rho$ in $L^p(X)$. 
\end{proof}

Modulus remains unchanged if we remove a zero modulus family of curves. The same property holds for our generalized modulus $\Mod_p(\Gamma,b)$. The proof follows directly from the previous Lemma and is the same as in the classical case \cite{Bjorn2011}*{Proposition 1.37}.

\begin{lemma} \label{lem:relaxation-of-admissibility}
Let $\Gamma\subset\Curv(X)$ and $b\colon\Gamma\to[-\infty,\infty]$. Then
\[ \Mod_p(\Gamma,b)=\inf\Big\{
\int_X\rho^pd\mu:\rho\in\Bor_+(X)
\text{ and }\int_\gamma\rho ds\geq b(\gamma)\text{ for $p$-a.e. }\gamma\in\Gamma\Big\}. \]
\end{lemma}

By a weak minimizer we mean a weakly admissible function $\rho\in\Bor_+(X)$ such that 
$\int_X\rho^pd\mu=\Mod_p(\Gamma,b)$. For the continuous version we require that $\int_X\rho^pd\mu=\Mod_p^c(\Gamma,b)$.

\begin{lemma} \label{lem:primal-minimizer}
Let $\Gamma\subset\Curv(X)$ and $b\colon\Gamma\to[-\infty,\infty]$.
    Suppose that $\Mod_p(\Gamma,b)<\infty$. Then there exists $\rho^*\in\LB^p_+(X)$ that is a unique (as a member of $L^p(X)$) weak minimizer of $\Mod_p(\Gamma,b)$. Moreover, for any minimizing sequence $\rho_k$ for $\Mod_p(\Gamma,b)$, we have $\rho_k\to \rho^*$ in $L^p(X)$ and thus $\rho^* \in \overline{\mathcal{A}(\Gamma,b)}$. Conversely, if $\rho^* \in \overline{\mathcal{A}(\Gamma,b)}$  and $\|\rho^*\|_{L^p(X)}^p = \Mod_p(\Gamma,b)$, then $\rho^*$ is a weak minimizer.
\end{lemma}

\begin{proof}
Let us take a minimizing sequence $(\rho_k)$. That is, $\rho_k\in\A(\Gamma,b)$ and
$$ \lim_{k\to\infty} \|\rho_k\|_{L^p(X)}=\inf_{\rho\in\A(\Gamma,b)}\|\rho\|_{L^p(X)}. $$
Since $(\rho_k)$ is a bounded sequence in $L^p(X)$, we may assume that it converges weakly in $L^p(X)$ to some $\rho^*\in L^p(X)$. 
By Mazur's lemma and the convexity of $\A(\Gamma,b)$ we may assume that in fact $\rho_k\xrightarrow{k\to\infty}\rho^*$ strongly in $L^p(X)$. Moreover, we can select a nonnegative Borel representative of $\rho^*$. Fuglede's lemma \ref{lem:Fuglede} implies that $\rho^*$ is a weak minimizer.
The uniqueness of $\rho^*$ follows from Clarkson's inequalities. 

The converse direction follows from the first part by choosing a sequence $\rho_k \in \mathcal{A}(\Gamma,b)$ converging to $\rho^*$.
\end{proof}

The following Lemma is an adaptation of the argument in \cite{HeinonenK98}*{Proposition 2.17}.
It says that the generalized modulus and its continuous version agree if the curve family is compact and the bounding function is continuous with a positive infimum.

\begin{lemma} \label{lem:continuous-densities}
Suppose that $X$ is compact.
Let $\Gamma\subset \Curv(X)$ be compact and $b\in C(\Gamma)$ such that $\inf_{\gamma\in\Gamma} b(\gamma)>0$. Then
$$ \Mod_p(\Gamma,b)=\Mod_p^c(\Gamma,b). $$
\end{lemma}

\begin{proof}
Clearly $\Mod_p(\Gamma,b)\leq\Mod_p^c(\Gamma,b)$. 
We show that $\Mod_p(\Gamma,b)\geq\Mod_p^c(\Gamma,b)$.

Let $\rho\in\A(\Gamma,b)$ and fix $\epsilon \in (0,\inf_{\gamma\in\Gamma} b(\gamma))$.
By the Vitali-Carath\'{e}odory theorem we can find a lower semi-continuous function $\Tilde{\rho}$ such that
\[ \rho\leq\Tilde{\rho}
\quad\text{and}\quad \|\Tilde{\rho}\|_{L^p(X)}\leq \|\rho\|_{L^p(X)}+\epsilon. \]
Moreover, by the lower semi-continuity of $\Tilde{\rho}$ we can find a sequence of nonnegative continuous functions $(\Tilde{\rho}_n)_{n=1}^\infty$ such that $\Tilde{\rho}_n\nearrow\Tilde{\rho}$.

We aim to show that the continuous function $\Tilde{\rho}_n$ is almost admissible for $\Gamma$ with respect to $b$, provided that $n$ is large enough. This can be achieved under the assumption that $\Gamma$ is compact and $b$ is continuous.
Let 
$$ a_n:
=\inf_{\gamma\in\Gamma}\Big(\int_\gamma\Tilde{\rho}_nds-b(\gamma)\Big)
=\inf_{\Gamma}\big(A\Tilde{\rho}_n-b\big). $$
Since $\Gamma$ is compact and $A\Tilde{\rho}_n-b$ is a lower semi-continuous function on $\Gamma$ (by Lemma \ref{lem:regularity-of-various-maps-on-Curv}), there exists $\gamma_n\in\Gamma$ such that
$$ a_n=\int_{\gamma_n}\Tilde{\rho}_nds-b(\gamma_n). $$
Moreover, we can find a subsequence of $(\gamma_n)_{n=1}^\infty$, still denoted by itself, and $\gamma\in\Gamma$ such that $\gamma_n\xrightarrow{n\to\infty}\gamma$.
Fix $k\in \N$. By using the lower semi-continuity of  $A\tilde{\rho}_k-b$ and the fact that $(\Tilde{\rho}_n)$ is an increasing sequence,
\[ \liminf_{n\to\infty}a_n
=\liminf_{n\to\infty}\Big(\int_{\gamma_n}\Tilde{\rho}_nds-b(\gamma_n)\Big)
\geq \liminf_{n\to\infty}\Big(\int_{\gamma_n}\Tilde{\rho}_kds-b(\gamma_n)\Big) 
\geq \int_{\gamma}\Tilde{\rho}_kds-b(\gamma). 
\]
This holds for all $k\in\N$, and by monotone convergence with $k\to\infty$,
\[
\liminf_{n\to\infty}a_n
\geq \int_{\gamma}\Tilde{\rho}ds-b(\gamma)\geq 0. 
\]
Now we have shown that there exists some index $N\in\N$ such that $a_N\geq-\epsilon$.

Finally, we denote $m:=\inf_\Gamma b>0$ and set $\tau=\frac{m}{m-\epsilon}\Tilde{\rho}_N$. Then $\tau\in\A(\Gamma,b)\cap C(X)$. Indeed, for any $\gamma\in\Gamma$
$$ \int_{\gamma}\tau ds
=\frac{m}{m-\epsilon}\int_\gamma \Tilde{\rho}_N ds
\geq\frac{m}{m-\epsilon}(b(\gamma)-\epsilon)
\geq b(\gamma). $$
Moreover, 
\begin{align*}
    \int_X\tau^pd\mu
    =\Big(\frac{m}{m-\epsilon}\Big)^p\|\Tilde{\rho}_N\|_{L^p(X)}^p
    \leq 
    \Big(\frac{m}{m-\epsilon}\Big)^p\|\Tilde{\rho}\|_{L^p(X)}^p
    \leq 
    \Big(\frac{m}{m-\epsilon}\Big)^p\Big(\int_X\rho^pd\mu+\epsilon\Big).
\end{align*}
It follows that 
\begin{align*}
    \Mod_p^c(\Gamma,b)
    \leq
    \Big(\frac{m}{m-\epsilon}\Big)^p\Big(\int_X\rho^pd\mu+\epsilon\Big).
\end{align*}
We let $\epsilon\to 0$ and take infimum over $\rho\in\A(\Gamma,b)$ to obtain that $\Mod_p^c(\Gamma,b)\leq\Mod_p(\Gamma,b)$.
\end{proof}

The following small lemma is needed in Section \ref{ssec:convergenge-of-etas}.

\begin{lemma} \label{lem:Mod-inequality}
Let $\Gamma\subset\Curv(X)$. If $b\colon\Gamma\to\R$ is bounded, then $\Mod_p(\Gamma,b)\leq \|b\|_\infty^p\Mod_p(\Gamma,1)$.
\end{lemma}

\begin{proof}
Let $\rho\in\A(\Gamma,1)$. Then $\tilde{\rho}:=\|b\|_\infty\rho\in\A(\Gamma,b)$, because
$$ \int_\gamma\tilde{\rho}ds=\|b\|_\infty\int_\gamma\rho ds\geq b(\gamma)
\quad\text{for all }\gamma\in\Gamma. $$
Thus
$$ \Mod_p(\Gamma,b)\leq \int_X\tilde{\rho}^pd\mu
=\|b\|_\infty^p\int_X\rho^pd\mu. $$
Take infimum over $\rho$ to obtain the desired inequality.
\end{proof}
\section{\texorpdfstring{$p$}{p}-harmonic functions and a general modulus problem} \label{sec:energy}

In this section we prove that the minimal $p$-weak upper gradient of a $p$-harmonic function with prescribed boundary data is a weak minimizer of an appropriate generalized modulus problem. This requires us to take a small detour through Dirichlet problems.

We follow \cite{Bjorn2011} where the standing assumptions in the metric measure space $X$ for the validity of potential theory are the following: $1<p<\infty$ and $X=(X,d,\mu)$ is a complete, doubling $p$-Poincar\'{e} space.
Let $\Om\subset X$ be a bounded domain. We assume that $\mu(\partial\Om)=0$, and we consider the problem of minimizing the $p$-energy among Newtonian functions on $\Om$ that meet a prescribed boundary value on $\partial\Om$. More precisely, fix a Borel function $f\in N^{1,p}(\overline{\Om})$. 
We study the $p$-Dirichlet problem with boundary data $f$, i.e. the minimization problem
\begin{equation} \label{eq:energy-minimization-open}
    \inf\Big\{\int_{\Om} g_v^pd\mu:v-f|_\Om\in N^{1,p}_0(\Om)\Big\}.
\end{equation}
Note that the boundary value $f$ is taken in the Sobolev sense, as in the definition of obstacle problem \cite{Bjorn2011}*{Definition 7.1}.
It is well-known \cite{Shanmugalingam2001} (or \cite{Bjorn2011}*{Theorem 7.2, Proposition 7.14, Theorem 8.14}) that this problem has a unique minimizer $u$ that is $p$-harmonic in $\Om$ in the sense of Definition \ref{def:p-harmonic}. Conversely, every $p$-harmonic function $u\in N^{1,p}_{\rm loc}(\Omega)$ solves the $p$-Dirichlet problem with boundary data $u$ on all subdomains $\Omega'\subset \Omega$ compactly contained in $\Omega$.

For the statement of the generalized modulus problem, let $\Gamma$ denote all the curves inside $\overline{\Om}$ that join two boundary points of $\partial\Om$,
\begin{equation} \label{eq:Gamma-for-equivalence-proposition}
    \Gamma:=\{\gamma\in\Curv(\overline{\Om}):\gamma(0),\gamma(\ell_\gamma)\in\partial\Om\}.    
\end{equation}
We call such curves boundary-to-boundary curves in $\overline{\Om}$. Let $b\colon\Gamma\to\R$ be given by 
\begin{equation} \label{eq:b-for-equivalence-proposition}
    b(\gamma):=|f(\gamma(0))-f(\gamma(\ell_\gamma))| \quad\text{for all }\gamma\in\Gamma.    
\end{equation}
In other words, $b(\gamma)$ is the endpoint variation of $\gamma$ with respect to the boundary value $f$.
Consider the generalized modulus problem
\begin{equation} \label{eq:generalized-modulus}
    \Mod_p(\Gamma,b)=\inf\Big\{\int_{\overline{\Om}}\rho^pd\mu:\rho\in\A(\Gamma,b)\Big\}.
\end{equation}
By Lemma \ref{lem:primal-minimizer} there exists a weak minimizer $\rho^*\in\LB^p_+(\overline{\Om})$ of \eqref{eq:generalized-modulus}, which is unique as a member of $L^p(\overline{\Om})$.

The following proposition is the main result of this section. Its proof is postponed at the end of the section.

\begin{proposition} \label{prop:equivalence}
Let $X=(X,d,\mu)$ be a complete, doubling $p$-Poincar\'{e} space, $1<p<\infty$.
Suppose that $\Om\subset X$ is a bounded domain such that $\mu(\partial\Om)=0$, and fix $f\in N^{1,p}(\overline{\Om})$.
Let $\Gamma$ denote all boundary-to-boundary curves in $\overline{\Om}$ as in \eqref{eq:Gamma-for-equivalence-proposition} and let $b\colon\Gamma\to[0,\infty]$ be given by \eqref{eq:b-for-equivalence-proposition}. If $u$ is the $p$-harmonic minimizer of the energy minimization problem \eqref{eq:energy-minimization-open} and $\rho^*$ is the weak minimizer of the generalized modulus problem \eqref{eq:generalized-modulus}, then $g_u=\rho^*$ $\mu$-a.e. in $\Om$.
\end{proposition}

This proposition is proved with the aid of the following lemma, which connects the generalized modulus problem and the minimization problem of $p$-energy with continuous boundary data taken in the pointwise sense.

\begin{lemma} \label{lem:m1-equal-with-m2}
Under the assumptions of Proposition \ref{prop:equivalence},
$$ \inf\Big\{\int_{\overline{\Om}} g_v^pd\mu:v\in N^{1,p}({\overline{\Om}}),v=f\text{ on }\partial\Om\Big\}
=\inf\Big\{\int_{\overline{\Om}}\rho^pd\mu:\rho\in\A(\Gamma,b)\Big\}. $$
\end{lemma}

\begin{proof}
Let us denote
$$ m_1:=\inf\Big\{\int_{\overline{\Om}} g_v^pd\mu:v\in N^{1,p}({\overline{\Om}}),v=f\text{ on }\partial\Om\Big\} 
\quad\text{and}\quad 
m_2:=\inf\Big\{\int_{\overline{\Om}}\rho^pd\mu:\rho\in\A(\Gamma,b)\Big\}. $$

First, we show that $m_2\leq m_1$. 
Let us fix $v\in N^{1,p}(\overline{\Om})$ such that $v=f$ on $\partial\Om$. (If such $v$ does not exist, then $m_1=\infty$, and there is nothing to prove.)
Then
$$ |f(\gamma(0))-f(\gamma(\ell_\gamma))|\leq \int_\gamma g_vds \quad\text{for $p$-a.e. }\gamma\in\Gamma. $$ 
Lemma \ref{lem:relaxation-of-admissibility} allows us to conclude that
$$ m_2\leq \int_{\overline{\Om}}(g_v)^pd\mu. $$
We take infimum over $v\in N^{1,p}(\overline{\Om})$ to conclude that $m_2\leq m_1$.

Next, we show that $m_1\leq m_2$. We may assume that $m_2<\infty$ because otherwise there is nothing to prove. In this case $\A(\Gamma,b)\cap\LB^p(\overline{\Om})$ is non-empty.
Let $\rho\in\A(\Gamma,b)\cap\LB^p(\overline{\Om})$.
We claim that there exists $v\in N^{1,p}(\overline{\Om})$ such that $v=f$ on $\partial\Om$ and $\rho$ is an upper gradient of $v$.
If we can prove this, then by the definition of the minimal $p$-weak upper gradient
$$ \int_{\overline{\Om}}g_v^pd\mu\leq\int_{\overline{\Om}}\rho^pd\mu. $$
It follows that
$$ m_1\leq\int_{\overline{\Om}}\rho^pd\mu
\quad\text{for all }\rho\in\A(\Gamma,b), $$
which implies that $m_1\leq m_2$.

We define the function $v$ as follows. For any $x\in\overline{\Om}$, let
$$ w(x):=\inf_{y\in\partial\Om}\Big(f(y)+\inf_\gamma\int_\gamma\rho ds\Big) $$
where the latter infimum is taken over all curves $\gamma\in\Curv(\overline{\Om})$ that join $y$ to $x$. 
Note that $w$ is measurable by the proof of \cite{Jarvenpaa2007}*{Corollary 1.10}. (The statement in \cite{Jarvenpaa2007} is slightly different, but the same Suslin-type argument applies by adding $f(y)$.) 
We set
$$ v(x):=\min\{w(x),\max_{\partial\Om}f\}. $$

We first study $w$. We check that $\rho$ is an upper gradient of $w$. Fix $x,y\in\overline{\Om}$. We aim to show that
$$ |w(x)-w(y)|\leq \int_\gamma\rho ds $$
for every curve $\gamma\in\Curv(\overline{\Om})$ that joins $x$ and $y$.
Without loss of generality we may assume that $w(x)\geq w(y)$. Let $y_0\in\partial\Om$ and let $\gamma_0$ be curve that joins $y_0$ to $y$. Let $\gamma$ be a curve that joins $y$ to $x$. Then by concatenating these curves and the definition of $w$ we get
\begin{align*}
    w(x)\leq f(y_0)+\int_{\gamma_0}\rho ds+\int_\gamma\rho ds.
\end{align*}
We can take infimum over $\gamma_0$ and $y_0$ to obtain that 
\begin{align*}
    w(x)\leq w(y)+\int_\gamma\rho ds
\end{align*}
for all curves $\gamma$ that join $y$ to $x$, which yields the desired inequality \eqref{eq:upper-gradient}.

Let us check that $w$ takes boundary value $f$. Fix $x\in\partial\Om$. Clearly, by taking a constant curve,
$$ w(x)=\inf_{y\in\partial\Om}\Big(f(y)+\inf_\gamma\int_\gamma\rho ds\Big)\leq f(x). $$
For the opposite inequality $w(x)\geq f(x)$, we employ the fact that $\rho\in\A(\Gamma,b)$ to find that
$$ |f(x)-f(y)|\leq \int_\gamma\rho ds $$
for all $y\in\partial\Om$ and for all curves $\gamma$ that join $x$ to $y$. Equivalently
$$ f(y)+|f(x)-f(y)|\leq f(y)+\int_\gamma\rho ds $$
We take infimum over $\gamma$ and $y$ to obtain
$$ \inf_{y\in\partial\Om}(f(y)+|f(x)-f(y)|)\leq w(x). $$
It follows that $f(x)\leq w(x)$, as desired. We conclude that $w(x)=f(x)$ for all $x\in\partial\Om$.

Now $w\colon\overline{\Om}\to [\min_{\partial\Om}f,\infty]$ is a measurable function such that $w=f$ on $\partial\Om$ and it has an upper gradient $\rho\in\LB^p(\overline{\Om})$. Consequently $v=\min\{w,\max_{\partial\Om}f\}$ is a Borel function that satisfies $v=f$ on $\partial\Om$ and $v$ has an upper gradient $\rho$, see \cite{Bjorn2011}*{Proof of Theorem 1.20}. Since $v$ is a bounded measurable function on a set of finite measure, we conclude that $v\in \LB^p(\overline{\Om})$ and moreover $v\in N^{1,p}(\overline{\Om})$.
\end{proof}

We now give the proof of Proposition \ref{prop:equivalence}.

\begin{proof}[Proof of Proposition \ref{prop:equivalence}]
Let us first show that the value of the energy minimization problem \eqref{eq:energy-minimization-open} equals the value of the generalized modulus problem \eqref{eq:generalized-modulus}.

Let $v\colon\Om\to[-\infty,\infty]$ be a function such that $v-f|_{\Om}\in N^{1,p}_0(\Om)$. 
By the definition of $N^{1,p}_0(\Om)$ in \eqref{eq:N1p-with-zero-boundary-value} we find an extension $\overline{v}\in N^{1,p}(\overline{\Om})$ of $v$ such that $\overline{v}=f$ on $\partial\Om$. Since $\mu(\partial\Om)=0$ we obtain that
\begin{equation*} 
    \int_{\Om} g_v^pd\mu
    =\int_{\overline{\Om}} g_{\overline{v}}^pd\mu
    \geq 
    \inf\Big\{\int_{\overline{\Om}} g_v^pd\mu:v\in N^{1,p}({\overline{\Om}}),v=f\text{ on }\partial\Om\Big\}.
\end{equation*}
We take infimum over $v$ to conclude that
\begin{equation} \label{eq:minimization-ineq} 
    \inf\Big\{\int_{\Om} g_v^pd\mu:v-f|_\Om\in N^{1,p}_0(\Om)\Big\}
    \geq 
    \inf\Big\{\int_{\overline{\Om}} g_v^pd\mu:v\in N^{1,p}({\overline{\Om}}),v=f\text{ on }\partial\Om\Big\}.
\end{equation}
For the converse inequality, let $v\in N^{1,p}(\overline{\Om})$ be such that $v=f$ on $\partial\Om$. Then by the definition of $N^{1,p}_0(\Om)$ in \eqref{eq:N1p-with-zero-boundary-value} we have that $(v-f)|_{\Om}\in N^{1,p}_0(\Om)$. Since $\mu(\partial\Om)=0$ we obtain that 
\begin{align*}
    \int_{\overline{\Om}}g_v^pd\mu
    =\int_{\Om}g_{v|_\Om}^pd\mu
    \geq     \inf\Big\{\int_{\Om} g_v^pd\mu:v-f|_\Om\in N^{1,p}_0(\Om)\Big\}.
\end{align*}
Take infimum over $v$ to obtain that
\begin{equation} \label{eq:minimization-converse}
    \inf\Big\{\int_{\overline{\Om}} g_v^pd\mu:v\in N^{1,p}({\overline{\Om}}),v=f\text{ on }\partial\Om\Big\}
    \geq
    \inf\Big\{\int_{\Om} g_v^pd\mu:v-f|_\Om\in N^{1,p}_0(\Om)\Big\}.
\end{equation}
Putting the above inequalities \eqref{eq:minimization-ineq} and \eqref{eq:minimization-converse} together with Lemma \ref{lem:m1-equal-with-m2} yields that
\begin{align*}
    \inf\Big\{\int_{\Om} g_v^pd\mu:v-f|_\Om\in N^{1,p}_0(\Om)\Big\}
    &= 
    \inf\Big\{\int_{\overline{\Om}} g_v^pd\mu:v\in N^{1,p}({\overline{\Om}}),v=f\text{ on }\partial\Om\Big\} \\
    &=
    \inf\Big\{\int_{\overline{\Om}}\rho^pd\mu:\rho\in\A(\Gamma,b)\Big\}.
\end{align*}
We conclude that
$$ \int_\Om g_u^pd\mu=\int_{\overline{\Om}}(\rho^*)^pd\mu. $$

To prove that $g_u=\rho^*$ we first observe that $g_u$ is weakly admissible for $\Gamma$. Indeed, $g_u$ is a $p$-weak upper gradient for $u$ and $u$ can be extended to a function $\overline{u}\in N^{1,p}(\overline{\Om})$ such that $\overline{u}=f$ on $\partial\Omega$. Then, the fact that $\Mod_p(\Gamma,b)=\int g_u^p d\mu$ together with the  uniqueness of the weak minimizer (Lemma \ref{lem:primal-minimizer}) implies that $g_u=\rho^*$ holds $\mu$-a.e. in $\Om$. 
\end{proof}
\section{Duality}\label{sec:dual}
In this section we apply the scheme of convex optimization to derive the dual problem of the generalized modulus problem. 
Throughout this section, let $X$ be a compact metric measure space.
\subsection{Primal problem and Lagrangian function}
Let $\Gamma\subset\Curv(X)$ be Borel and $b\in\Bor(\Gamma)$ be bounded.
Our primal problem is the minimization problem
$$ \big(\Mod_p(\Gamma,b)\big)^{1/p}=\inf_{\rho\in\A(\Gamma,b)}\|\rho\|_{L^p(X)} $$
or the continuous version
$$ \big(\Mod_p^c(\Gamma,b)\big)^{1/p}=\inf_{\rho\in\A(\Gamma,b)\cap C(X)}\|\rho\|_{L^p(X)}. $$
The purpose of this section is to prove that these problems are equivalent to finding a saddle point of Lagrangian function $\Lag\colon \LB^p_+(X)\times \M_+(\Gamma)\to[-\infty,\infty)$, given by
\begin{align} \label{eq:Lagrangian-new}
    \Lag(\rho,\eta):=\|\rho\|_{L^p(X)}-\int_{\Gamma}(\int_\gamma\rho ds-b(\gamma))d\eta. 
\end{align}
We point out that we may rewrite Lagrangian \eqref{eq:Lagrangian-new} by using the notation introduced in Section \ref{ssec:A-and-its-transpose} as
\begin{align*}
    \Lag(\rho,\eta)
    &=\|\rho\|_{L^p(X)}-\la A\rho-b,\eta\ra \\
    &=\|\rho\|_{L^p(X)}-\la\rho,A^\intercal\eta\ra+\la b,\eta\ra. 
\end{align*}

First we state an auxiliary lemma characterizing admissibility in terms of the Lagrangian.

\begin{lemma} \label{lem:primal-penalty}
Let $X$ be compact and $\Gamma\subset \Curv(X)$ be Borel and $b\in\Bor(\Gamma)$ be bounded. Then for any $\rho\in\mathfrak{B}_+(X)$
\begin{equation} \label{eq:primal-penalty}
    \sup_{\eta\in\M_+(\Gamma)}(-\la A\rho-b,\eta\ra)=
    \begin{cases}
        0&\quad\text{if }\rho\in\A(\Gamma,b); \\
        \infty &\quad\text{otherwise}.
    \end{cases}
\end{equation}
\end{lemma}

\begin{proof}
If $\rho\in\A(\Gamma,b)$, then $A\rho\geq b$. Let $\eta_0$ denote the zero measure. Then 
\begin{align*}
    0=-\la A\rho-b,\eta_0\ra
    \leq\sup_{\eta\in\M_+(\Gamma)}(-\la A\rho-b,\eta\ra)\leq 0.
\end{align*}
On the other hand, if $\rho\in\mathfrak{B}_+(X)\setminus\A(\Gamma,b)$ we can find $\gamma\in\Gamma$ such that $~(A\rho)(\gamma)<b(\gamma)$. Let $\delta_\gamma$ denote the Dirac delta measure at $\gamma$. Then 
\begin{align*}
    \sup_{\eta\in\M_+(\Gamma)}(-\la A\rho-b,\eta\ra)
    \geq \sup_{n\in\N}(-\la A\rho-b,n\delta_\gamma\ra)
    =\sup_{n\in\N}\big(-n\big((A\rho)(\gamma)-b(\gamma)\big)\big)
    =\infty.
\end{align*}
We conclude that \eqref{eq:primal-penalty} holds.
\end{proof}

The previous lemma allows us to formulate the primal problems in terms of Lagrangian function $\Lag$.

\begin{corollary} \label{cor:primal-both}
Let $X$ be compact and $\Gamma\subset \Curv(X)$ be Borel and $b\in\Bor(\Gamma)$ be bounded. We have 
$$ \inf_{\rho\in\A(\Gamma,b)}\|\rho\|_{L^p(X)}
=\inf_{\rho\in \LB^p_+(X)}\sup_{\eta\in\M_+(\Gamma)}\Lag(\rho,\eta). $$
and 
$$ \inf_{\rho\in\A(\Gamma,b)\cap C(X)}\|\rho\|_{L^p(X)}
=\inf_{\rho\in C_+(X)}\sup_{\eta\in\M_+(\Gamma)}\Lag(\rho,\eta). $$
\end{corollary}

\begin{proof}
For the first equality, by Lemma \ref{lem:primal-penalty} it is easy to see that
\begin{align*}
    \inf_{\rho\in\A(\Gamma,b)}\|\rho\|_{L^p(X)} 
    &=
    \inf_{\rho\in\LB^p_+(X)}\Big(\|\rho\|_{L^p(X)}+\sup_{\eta\in\M_+(\Gamma)}(-\la A\rho-b,\eta\ra)\Big) \\
    &=
    \inf_{\rho\in\LB^p_+(X)}\sup_{\eta\in\M_+(\Gamma)}\Lag(\rho,\eta).
\end{align*}
The proof of the second equality is the same, mutatis mutandis.
\end{proof}

The following lemma says that under some additional assumptions for the continuous modulus problem we can restrict the set of measures over which we maximize the Lagrangian.
This restriction is crucial to obtain strong duality in Section \ref{ssec:strong-duality-new} below.

\begin{lemma} \label{lem:primal}
Let $X$ be compact and $\Gamma\subset \Curv(X)$ be Borel and $b\in\Bor(\Gamma)$ be bounded. Suppose that \\
$0<\Mod_p^c(\Gamma,b)<\infty$ and $\inf_\Gamma b>0$.
Then
$$ \inf_{\rho\in\A(\Gamma,b)\cap C(X)}\|\rho\|_{L^p(X)}
=\inf_{\rho\in C_+(X)}\sup_{\stackrel{\eta\in\M_+(\Gamma)}{\eta(\Gamma)\leq M}} \Lag(\rho,\eta) $$
where 
$$ M:=\frac{\big(\Mod_p^c(\Gamma,b)\big)^{1/p}}{\inf_{\Gamma}b}. $$
\end{lemma}

\begin{proof}
It follows from Corollary \ref{cor:primal-both} that
\begin{align*}
    \inf_{\rho\in\A(\Gamma,b)\cap C(X)}\|\rho\|_{L^p(X)} 
    \geq
    \inf_{\rho\in C_+(X)}\sup_{\stackrel{\eta\in\M_+(\Gamma)}{\eta(\Gamma)\leq M}}\Lag(\rho,\eta).
\end{align*}
In order to prove the converse inequality we show that for all $\rho\in C_+(X)$ 
\begin{equation} \label{eq:a-goal-primal}
    \inf_{\tau\in\A(\Gamma,b)\cap C(X)}\|\tau\|_{L^p(X)}\leq\sup_{\stackrel{\eta\in\M_+(\Gamma)}{\eta(\Gamma)\leq M}}\Lag(\rho,\eta).
\end{equation}
Let us denote
$$ E:=\inf_{\tau\in\A(\Gamma,b)\cap C(X)}\|\tau\|_{L^p(X)}. $$
Fix $\rho\in C_+(X)$. We consider three cases, $\|\rho\|_{L^p(X)}\geq E$, $0<\|\rho\|_{L^p(X)}<E$ and $\|\rho\|_{L^p(X)}=0$.

In the first case $\|\rho\|_{L^p(X)}\geq E$. Let $\eta_0$ denote the zero measure. Then
$$ E\leq \|\rho\|_{L^p(X)}-\la A\rho-b,\eta_0\ra \leq \sup_{\stackrel{\eta\in\M_+(\Gamma)}{\eta(\Gamma)\leq M}}\Lag(\rho,\eta) $$
This yields the desired inequality \eqref{eq:a-goal-primal} in the first case.

In the second case $0<\|\rho\|_{L^p(X)}<E$ and thus $\rho\notin\A(\Gamma,b)$. Moreover, given any $\epsilon>0$ small, if \\
$\lambda_\epsilon=\frac{E}{(1+\epsilon)\|\rho\|_{L^p(X)}}$ then $\lambda_\epsilon\rho\notin\A(\Gamma,b)$.
Consequently there exists a curve $\gamma_\epsilon\in\Gamma$ such that
$$ \int_{\gamma_\epsilon}\rho ds<\frac{(1+\epsilon)\|\rho\|_{L^p(X)}b(\gamma_\epsilon)}{E}. $$
Let $\delta_{\gamma_\epsilon}$ denote the Dirac delta measure at $\gamma_\epsilon$ and let $\eta_\epsilon:=\frac{E}{b(\gamma_\epsilon)}\delta_{\gamma_\epsilon}$.
Note that $b(\gamma_\epsilon)>0$ because $\inf_\Gamma b>0$.
Then $\eta_\epsilon(\Gamma)\leq M$ and
\begin{align*}
    \sup_{\stackrel{\eta\in\M_+(\Gamma)}{\eta(\Gamma)\leq M}}\Lag(\rho,\eta) 
    \geq 
    \Lag(\rho,\eta_\epsilon)
    =\|\rho\|_{L^p(X)}-\frac{E}{b(\gamma_\epsilon)}\big(\int_{\gamma_\epsilon}\rho ds-b(\gamma_\epsilon)\big)
    \geq -\epsilon\|\rho\|_{L^p(X)}+E.
\end{align*}
Let $\epsilon\to0$ to conclude the desired inequality \eqref{eq:a-goal-primal} in the second case.

In the third case $\|\rho\|_{L^p(X)}=0$ and thus $\rho=0$ everywhere in $X$ by the continuity of $\rho$.
Let us fix any curve $\gamma\in \Gamma$ and set $\Tilde{\eta}:=\frac{E}{b(\gamma)}\delta_\gamma$. Then $\Tilde{\eta}(\Gamma)\leq M$ and
\begin{align*}
    \sup_{\stackrel{\eta\in\M_+(\Gamma)}{\eta(\Gamma)\leq M}}\Lag(\rho,\eta) 
    \geq \Lag(\rho,\Tilde{\eta})=E. 
\end{align*}
This allows us to conclude the desired inequality \eqref{eq:a-goal-primal} in the third case.

Having proved \eqref{eq:a-goal-primal} in all three cases we conclude that it holds for any $\rho\in C_+(X)$. The proof is finished.
\end{proof}
\subsection{Dual problem}
Let $\Gamma\subset\Curv(X)$ be Borel and $b\in\Bor(\Gamma)$ be bounded. Given the Lagrangian formulation, the dual problem can be obtained by flipping the order of the supremum and infimum. 
Thus, from Corollary \ref{cor:primal-both} we get that the dual problem of the generalized modulus problem is to find
\begin{align*}
    \sup_{\eta\in\M_+(\Gamma)}\inf_{\rho\in\LB^p_+(X)}\Lag(\rho,\eta)
    &=\sup_{\eta\in\M_+(\Gamma)}\big(\la b,\eta\ra+\inf_{\rho\in\LB^p_+(X)}\big(\|\rho\|_{L^p(X)}-\la \rho,A^\intercal\eta\ra\big)\big).
\end{align*}
If in addition $\Gamma$ is compact and $\inf_\Gamma b>0$ 
then by Lemma \ref{lem:primal} the dual problem of the continuous modulus problem is
\begin{align*}
    \sup_{\stackrel{\eta\in\M_+(\Gamma)}{\eta(\Gamma)\leq M}}\inf_{\rho\in C_+(X)}\Lag(\rho,\eta)
    &=\sup_{\stackrel{\eta\in\M_+(\Gamma)}{\eta(\Gamma)\leq M}}\big(\la b,\eta\ra+\inf_{\rho\in C_+(X)}\big(\|\rho\|_{L^p(X)}-\la \rho,A^\intercal\eta\ra\big)\big).
\end{align*}
We show that these dual problems, that are formulated as a problem of finding a saddle point of Lagrangian $\Lag$, have equivalent formulations as maximization problems over a suitable subset of measures in $\M_+(\Gamma)$. See Lemma \ref{lem:dual-penalty} and Corollary \ref{cor:dual-both} below. Moreover, we find a maximizer of the dual problem. See Lemma \ref{lem:dual-maximizer} below.

The "suitable subset of measures" is given by the following definition.

\begin{definition} \label{def:B}
Let $X$ be compact and let $\Gamma\subset\Curv(X)$ and $1<q<\infty$. We define 
$$ \B(\Gamma):=\Big\{\eta\in\M_+(\Gamma):\Big\|\frac{d(A^\intercal\eta)}{d\mu}\Big\|_{L^q(X)}\leq 1\Big\}. $$
\end{definition}

\begin{remark}
Using the terminology in \cite{Savare2022}*{Definition 4.2.2}, the measures in the class $\B(\Gamma)$ are called dynamic plans with barycenter in the unit ball of $L^q(X)$.
\end{remark}

In the following lemma we record the fact that the collection $\B(\Gamma)$ is closed with respect to weak convergence.

\begin{lemma} \label{lem:B-closed-wrt-weak-convergence-to-Radon-measure}
Let $X$ be compact and $\Gamma\subset\Curv(X)$. If $(\eta_n)$ is a sequence of measures in $\B(\Gamma)$ that converges weakly to some $\eta\in\M_+(\Gamma)$, then $\eta\in\B(\Gamma)$.
\end{lemma}

\begin{proof}
The proof follows immediately from Lemma \ref{lem:B-weakly-closed-and-double-limit-in-dual-pairing}, item (1).
\end{proof}

\begin{remark} \label{rem:small-lemma}
It is useful to record the following observation: If $\eta\in\B(\Gamma)$ and $E\subset\Gamma$ is Borel, then $\eta\resmes E\in\B(E)$. See Remark \ref{rem:restriction-and-AT}.
\end{remark}

The following auxiliary lemma says that indeed, the aforementioned "suitable subset of measures" is precisely $\B(\Gamma)$.

\begin{lemma} \label{lem:dual-penalty}
Let $X$ be compact and $\Gamma\subset\Curv(X)$ Borel and $b\in\Bor(\Gamma)$ bounded. Then for any $\eta\in\M_+(\Gamma)$
\begin{equation} \label{eq:dual-penalty}
    \inf_{\rho\in C_+(X)}\big(\|\rho\|_{L^p(X)}-\la \rho,A^\intercal\eta\ra\big)
    =\inf_{\rho\in \LB^p_+(X)}\big(\|\rho\|_{L^p(X)}-\la \rho,A^\intercal\eta\ra\big)
    =
    \begin{cases}
        0&\quad\text{if }\eta\in\B(\Gamma); \\
        -\infty &\quad\text{otherwise}.
    \end{cases}
\end{equation}
\end{lemma}

\begin{proof}
We first prove the second equality in \eqref{eq:dual-penalty}. Let $\eta\in\M_+(\Gamma)$. We consider several cases.

First suppose that $A^\intercal\eta$ is not absolutely continuous with respect to $\mu$. In this case we can find a Borel set $E\subset X$ such that $\mu(E)=0$ but $(A^\intercal\eta)(E)>0$. Then
\begin{align*}
    \inf_{\rho\in\LB^p_+(X)}\big(\|\rho\|_{L^p(X)}-\la \rho,A^\intercal\eta\ra\big)
    \leq \inf_{n\in\N}\big(\|n\mathbbm{1}_E\|_{L^p(X)}-\la n\mathbbm{1}_E,A^\intercal\eta\ra\big)
    =-\infty.
\end{align*}

Next we suppose that $A^\intercal\eta$ is absolutely continuous with respect to $\mu$ and consequently
\begin{align*}
    \inf_{\rho\in\LB^p_+(X)}\big(\|\rho\|_{L^p(X)}-\la \rho,A^\intercal\eta\ra\big)
    =\inf_{\rho\in\LB^p_+(X)}\Big(\|\rho\|_{L^p(X)}-\int_X\rho\frac{d(A^\intercal\eta)}{d\mu}d\mu\Big)
\end{align*}
where $\frac{d(A^\intercal\eta)}{d\mu}$ denotes the Radon-Nikodym derivative of $A^\intercal\eta$ with respect to $\mu$.
We consider two cases, depending on whether or not $\frac{d(A^\intercal\eta)}{d\mu}$ lies in $L^q(X)$.

If $\frac{d(A^\intercal\eta)}{d\mu}$ does not lie in $L^q(X)$, then the functional $\rho\mapsto\int_X\rho\frac{d(A^\intercal\eta)}{d\mu}d\mu$ must be unbounded on $L^p(X)$. Thus, we can find a sequence $(\rho_n)$ of nonnegative functions in $L^p(X)$, that we may assume to be in $\LB^p(X)$, such that $\|\rho_n\|_{L^p(X)}=1$ and $\int_X\rho_n\frac{d(A^\intercal\eta)}{d\mu}d\mu\xrightarrow{n\to\infty}\infty$. 
Then
\begin{align*}
    \inf_{\rho\in\LB^p_+(X)}\big(\|\rho\|_{L^p(X)}-\la \rho,A^\intercal\eta\ra\big)
    \leq \inf_{n\in\N}\Big(\|\rho_n\|_{L^p(X)}-\int_X\rho_n\frac{d(A^\intercal\eta)}{d\mu}d\mu\Big)
    =-\infty.
\end{align*}
If $\frac{d(A^\intercal\eta)}{d\mu}$ lies in $L^q(X)$, and $\|\frac{d(A^\intercal\eta)}{d\mu}\|_{L^q(X)}>1$, then $\tau=(\frac{d(A^\intercal\eta)}{d\mu})^{q-1}$ lies in $L^p(X)$. We may choose a nonnegative Borel representative of $\tau$ so that $\tau\in\LB^p_+(X)$. Then
\begin{align*}
    \inf_{\rho\in\LB^p_+(X)}\big(\|\rho\|_{L^p(X)}-\la \rho,A^\intercal\eta\ra\big)
    &\leq \inf_{n\in\N}\Big(\|n\tau\|_{L^p(X)}-\int_Xn\tau\frac{d(A^\intercal\eta)}{d\mu}d\mu\Big) \\
    &=\inf_{n\in\N}\Big(n\Big(\Big\|\frac{d(A^\intercal\eta)}{d\mu}\Big\|_{L^q(X)}^{q/p}
    -\Big\|\frac{d(A^\intercal\eta)}{d\mu}\Big\|_{L^q(X)}^q\Big)\Big)
    =-\infty.
\end{align*}
Finally, if $\frac{d(A^\intercal\eta)}{d\mu}$ lies in $L^q(X)$, and $\|\frac{d(A^\intercal\eta)}{d\mu}\|_{L^q(X)}\leq 1$, then 
\begin{align*}
    0
    &\geq\inf_{\rho\in\LB^p_+(X)}\big(\|\rho\|_{L^p(X)}-\la \rho,A^\intercal\eta\ra\big)
    =\inf_{\rho\in\LB^p_+(X)}\Big(\|\rho\|_{L^p(X)}-\int_X\rho\frac{d(A^\intercal\eta)}{d\mu}d\mu\Big) \\
    &\geq\inf_{\rho\in\LB^p_+(X)}\Big(\|\rho\|_{L^p(X)}\Big(1-\Big\|\frac{d(A^\intercal\eta)}{d\mu}\Big\|_{L^q(X)}\Big)\Big)
    =0.
\end{align*}
Above we employed Hölder's inequality.

Having studied all the possible cases for $\eta\in\M_+(\Gamma)$ we conclude that the second equality in \eqref{eq:dual-penalty} holds. It remains to check the first equality in \eqref{eq:dual-penalty}, that is,
$$ \inf_{\rho\in C_+(X)}\big(\|\rho\|_{L^p(X)}-\la \rho,A^\intercal\eta\ra\big)=\inf_{\rho\in \LB^p_+(X)}\big(\|\rho\|_{L^p(X)}-\la \rho,A^\intercal\eta\ra\big). $$
It suffices to check that
\begin{align} \label{eq:a-goal-dual}
    \inf_{\rho\in C_+(X)}\big(\|\rho\|_{L^p(X)}-\la \rho,A^\intercal\eta\ra\big)\leq\inf_{\rho\in \LB^p_+(X)}\big(\|\rho\|_{L^p(X)}-\la \rho,A^\intercal\eta\ra\big)
\end{align}
since the opposite inequality is trivial.

Let $\rho\in\LB^p_+(X)$ and $\epsilon>0$. By Vitali-Carath\'{e}odory theorem we can find a lower semi-continuous function $\Tilde{\rho}$ such that
$$ \rho\leq\Tilde{\rho}
\quad\text{and}\quad \|\Tilde{\rho}\|_{L^p(X)}\leq \|\rho\|_{L^p(X)}+\epsilon. $$
By the lower semi-continuity of $\Tilde{\rho}$ we can find a sequence of nonnegative continuous functions $(\Tilde{\rho}_n)$ such that $\Tilde{\rho}_n\nearrow\Tilde{\rho}$.
Monotone convergence theorem yields that
$$ \lim_{n\to\infty}(\|\Tilde{\rho}_n\|_{L^p(X)}-\la\Tilde{\rho}_n,A^\intercal\eta\ra)=\|\Tilde{\rho}\|_{L^p(X)}-\la\Tilde{\rho},A^\intercal\eta\ra. $$
Then
\begin{align*}
    \inf_{\tau\in C_+(X)}\big(\|\tau\|_{L^p(X)}-\la \tau,A^\intercal\eta\ra\big)
    &\leq \lim_{n\to\infty}(\|\Tilde{\rho}_n\|_{L^p(X)}-\la\Tilde{\rho}_n,A^\intercal\eta\ra) \\
    &=\|\Tilde{\rho}\|_{L^p(X)}-\la\Tilde{\rho},A^\intercal\eta\ra \\
    &\leq \|\rho\|_{L^p(X)}+\epsilon-\la \rho,A^\intercal\eta\ra.
\end{align*}
Take an infimum over $\rho\in\LB^p_+(X)$ and let $\epsilon\to0$ in order to arrive at the desired inequality \eqref{eq:a-goal-dual}. The proof is finished.
\end{proof}

\begin{corollary} \label{cor:dual-both}
Let $X$ be compact and $\Gamma\subset\Curv(X)$ and $b\in\Bor(\Gamma)$ bounded. Suppose that $0<M\leq \infty$.
Then
$$ \sup_{\eta\in\B(\Gamma)}\int_\Gamma b d\eta=\sup_{\eta\in\M_+(\Gamma)}\inf_{\rho\in\LB^p_+(X)}\Lag(\rho,\eta) $$
and
$$ \sup_{\stackrel{\eta\in\B(\Gamma)}{\eta(\Gamma)\leq M}}\int_\Gamma b d\eta
=\sup_{\stackrel{\eta\in\M_+(\Gamma)}{\eta(\Gamma)\leq M}}\inf_{\rho\in C_+(X)}\Lag(\rho,\eta). $$
\end{corollary}

\begin{proof}
For the first equality, by Lemma \ref{lem:dual-penalty}
\begin{align*}
    \sup_{\eta\in\B(\Gamma)}\int_\Gamma b d\eta
    &=
    \sup_{\eta\in\M_+(\Gamma)}\Big(\la b,\eta\ra+\inf_{\rho\in\LB^p_+(X)}(\|\rho\|_{L^p(X)}-\la\rho,A^\intercal\eta\ra)\Big) \\
    &=
    \sup_{\eta\in\M_+(\Gamma)}\inf_{\rho\in\LB^p_+(X)}\Lag(\rho,\eta).
\end{align*}
The proof of the second equality is analogous.
\end{proof}

Corollary \ref{cor:dual-both} allows us to conclude that the dual problem of the generalized modulus problem is the maximization problem of
$$ \eta\mapsto \int_\Gamma bd\eta $$
over measures $\eta\in\B(\Gamma)$.
The following lemma says that if $\Gamma$ is compact and $b$ is continuous, we can guarantee the existence of a maximizer in the admissible class $\B(\Gamma)\cap\{\eta\in\M_+(\Gamma):\eta(\Gamma)\leq M\}$.

\begin{lemma} \label{lem:dual-maximizer}
Let $X$ and $\Gamma\subset \Curv(X)$ be compact and $b\in C(X)$. Suppose that $0<M<\infty$.
Then there exists $\eta^*\in\B(\Gamma)$ such that $\eta^*(\Gamma)\leq M$ and
$$ \int_\Gamma bd\eta^*=\sup_{\stackrel{\eta\in\B(\Gamma)}{\eta(\Gamma)\leq M}}\int_\Gamma b d\eta. $$
\end{lemma}

\begin{proof}
Let us take a maximizing sequence $(\eta_n)$. That is, $\eta_n\in\B$, $\eta_n(\Gamma)\leq M$ and
$$ \lim_{n\to\infty}\int_\Gamma bd\eta_n=\sup_{\stackrel{\eta\in\B}{\eta(\Gamma)\leq M}}\int_\Gamma b d\eta. $$
By Prokhorov's theorem, Theorem \ref{thm:prokhorov}, we may assume that $(\eta_n)$ converges weakly to some $\eta^*\in\M_+(\Gamma)$. Thus $\eta_n(\Gamma)\xrightarrow{n\to\infty}\eta^*(\Gamma)$ and consequently $\eta^*(\Gamma)\leq M$. Moreover, since $b\in C(\Gamma)$,
$$ \lim_{n\to\infty}\int_\Gamma bd\eta_n=\int_\Gamma bd\eta^*. $$
Thus $\eta^*$ reaches the desired supremum. 
Lemma \ref{lem:B-closed-wrt-weak-convergence-to-Radon-measure} implies that $\eta^*\in\B(\Gamma)$.
\end{proof}
\subsection{Weak duality for \texorpdfstring{$\Mod_p$}{Modp}}

We need the following two auxiliary results in the proof of the main theorem in Section \ref{sec:proof-of-main}. This fact is often called \emph{weak duality}.

\begin{lemma} \label{lem:weak-duality}
Let $X$ be compact and $\Gamma\subset\Curv(X)$ be Borel and $b\in\Bor(\Gamma)$ be bounded. Then
$$ \sup_{\eta\in\B(\Gamma)}\int_\Gamma bd\eta\leq
\inf_{\rho\in\A(\Gamma,b)}\|\rho\|_{L^p(X)}. $$
In particular, if $\eta\in\B(\Gamma)$ then $\int_\Gamma bd\eta\leq(\Mod_p(\Gamma,b))^{1/p}$.
\end{lemma}

\begin{proof}
We may assume that $\inf_{\rho\in\A(\Gamma,b)}\|\rho\|_{L^p(X)}<\infty$, because otherwise the claim follows trivially.  

By Corollary \ref{cor:primal-both}
\begin{align*}
    \inf_{\rho\in\A(\Gamma,b)}\|\rho\|_{L^p(X)} 
    &=
    \inf_{\rho\in \LB^p_+(X)}\sup_{\eta\in\M_+(\Gamma)}\Lag(\rho,\eta),
\end{align*}
and by Corollary \ref{cor:dual-both}
\begin{align}
    \sup_{\eta\in\B}\int_\Gamma b d\eta
    &=
    \sup_{\eta\in\M_+(\Gamma)}\inf_{\rho\in \LB^p_+(X)}\Lag(\rho,\eta). \label{eq:maxminineq}
\end{align}
By max-min inequality
$$ \sup_{\eta\in\M_+(\Gamma)}\inf_{\rho\in \LB^p_+(X)}\Lag(\rho,\eta)
\leq\inf_{\rho\in \LB^p_+(X)}\sup_{\eta\in\M_+(\Gamma)}\Lag(\rho,\eta), $$
from which the claim follows.
\end{proof}

From the previous lemma with $b\equiv 1$ we obtain the following corollary.

\begin{corollary} \label{cor:p-exceptional-is-eta-nul}
Let $X$ be compact and $\Gamma\subset\Curv(X)$ be a $p$-exceptional Borel set. 
Then for any $\eta\in\B(\Gamma)$ it holds that $\eta(\Gamma)=0$.
\end{corollary}
\subsection{Strong duality for \texorpdfstring{$\Mod_p^c$}{Modpc}} \label{ssec:strong-duality-new}
The inequality in Lemma \ref{lem:weak-duality} becomes an equality thanks to the minimax theorem, which replaces the inequality in \eqref{eq:maxminineq} with an equality. This is called strong duality. In this section, we apply the following version of the  minimax theorem \cite{Sion1958}*{Theorem 4.2}, which is due to Kneser \cite{Kneser} and Fan \cite{Fan} (but is sometimes called Sion's minimax theorem). These are all generalizations of the famous minimax theorem of Von Neumann \cite{VonNeumann}. 

\begin{theorem}[Kneser-Fan minimax theorem \cite{Sion1958}*{Theorem 4.2}] \label{thm:sion}
Let $A$ be a convex subset of a linear topological space and let $B$ be a compact and convex subset of a linear topological space. Suppose that $f\colon A\times B\to \R$ is a function that is
\begin{enumerate}
    \item convex in the first variable: $f(\cdot,b)\colon A\to \R$ is convex for all $b\in B$;
    \item upper semi-continuous and concave in the second variable: $f(a,\cdot)\colon B\to \R$ is upper semi-continuous and concave for all $a\in A$.
\end{enumerate}
Then
$$ \inf_{a\in A}\sup_{b\in B} f(a,b)=\sup_{b\in B}\inf_{a\in A}f(a,b). $$
\end{theorem}

\begin{proposition} \label{prop:strong-duality}
Let $X$ and $\Gamma\subset \Curv(X)$ be compact and $b\in C(X)$. 
Suppose that 
$0<\Mod_p^c(\Gamma,b)<\infty$ and $\inf_\Gamma b>0$.
Then
\begin{align} \label{eq:strong-duality}
    \inf_{\rho\in\A(\Gamma,b)\cap C(X)}\|\rho\|_{L^p(X)}
    =\inf_{\rho\in C_+(X)}\sup_{\stackrel{\eta\in\M_+(\Gamma)}{\eta(\Gamma)\leq M}}\Lag(\rho,\eta) 
    =\sup_{\stackrel{\eta\in\M_+(\Gamma)}{\eta(\Gamma)\leq M}}\inf_{\rho\in C_+(X)}\Lag(\rho,\eta)
    =\sup_{\stackrel{\eta\in\B}{\eta(\Gamma)\leq M}}\int_\Gamma b d\eta.
\end{align}
where
$$ M:=\frac{\big(\Mod_p^c(\Gamma,b)\big)^{1/p}}{\inf_{\Gamma}b}. $$
\end{proposition}

\begin{proof}
The first and third equalities in display \eqref{eq:strong-duality} follow from Lemma \ref{lem:primal} and Corollary \ref{cor:dual-both} respectively.
The second equality follows from the minimax theorem, Theorem \ref{thm:sion}, applied to the Lagrangian function
$$ \Lag\colon C_+(X)\times \{\eta\in\M_+(\Gamma):\eta(\Gamma)\leq M\}\to\R, $$
given by
$$ \Lag(\rho,\eta)=\|\rho\|_{L^p(X)}-\la A\rho-b,\eta\ra. $$
We check that the assumptions of the minimax theorem are satisfied in this setting. 

Note that the sets $C_+(X)$ and $\{\eta\in\M_+(\Gamma):\eta(\Gamma)\leq M\}$ are convex and the latter is compact with respect to the weak topology of Radon measures by Prokhorov's theorem \ref{thm:prokhorov}. 
It is clear that $\Lag$ is convex in the first variable and concave (in fact, affine) in the second variable. 
For the upper semi-continuity of $\Lag$ in the second variable, note that $\Lag(\rho,\cdot)$ is an affine function and since $\eta\to \la A\rho-b,\eta\ra \in  LSC(\Gamma)$ by Lemma \ref{lem:lsc-of-dual-pairing} and since $b$ is continuous. 
\end{proof}

The following proposition is the main conclusion of this section.

\begin{proposition} \label{prop:properties}
Let $X$ and $\Gamma\subset \Curv(X)$ be compact and $b\in C(\Gamma)$. Suppose 
$0<\Mod_p^c(\Gamma,b)<\infty$ and $\inf_\Gamma b>0$, and let
$$ M:=\frac{\big(\Mod_p^c(\Gamma,b)\big)^{1/p}}{\inf_{\Gamma}b}. $$
Then there exist $\rho^*\in\LB^p_+(X)$ and $\eta^*\in\M_+(\Gamma)$ such that the following hold:
\begin{enumerate}
    \item $\int_\gamma\rho^*ds\geq b(\gamma)$ for $p$-a.e. $\gamma\in\Gamma$.
    \item $\eta^*(\Gamma)\leq M$ and $\eta^*\in\B(\Gamma)$.
    \item $(\Mod_p^c(\Gamma,b))^{1/p}=\|\rho^*\|_{L^p(X)}=\int_\Gamma bd\eta^*$.
    \item We have that
    $$ \int_\Gamma \int_\gamma\rho^*ds-b(\gamma)d\eta^*=0. $$ 
    \item Measure $\eta^*$ is supported on those curves $\gamma\in\Gamma$ for which $\int_\gamma\rho^*ds=b(\gamma)$. In other words, $\int_\gamma\rho^*ds=b(\gamma)$ holds for $\eta^*$-a.e. curve $\gamma\in\Gamma$.
    \item For $\mu$-a.e.
    \[\frac{d(A^\intercal\eta^*)}{d\mu}=\left(\frac{\rho^*}{\|\rho^*\|_{L^p(X)}}\right)^{p-1}.\]
    \item For any $\rho\in\Bor_+(X)$ we have that
    $$ \int_\Gamma\int_\gamma\rho dsd\eta^*=\int_X\rho\left(\frac{\rho^*}{\|\rho^*\|_{L^p(X)}}\right)^{p-1}d\mu. $$
\end{enumerate}    
\end{proposition}

\begin{proof}
Let $\rho^*$ and $\eta^*$ be given by Lemma \ref{lem:primal-minimizer} and Lemma \ref{lem:dual-maximizer}, respectively. Then properties (1) and (2) are immediate. For property (3) we employ Proposition \ref{prop:strong-duality} to find that
\begin{align} \label{eq:strong-duality-in-proof}
    \|\rho^*\|_{L^p(X)}
    =\inf_{\rho\in C_+(X)}\sup_{\stackrel{\eta\in\M_+(\Gamma)}{\eta(\Gamma)\leq M}}\Lag(\rho,\eta) 
    =\sup_{\stackrel{\eta\in\M_+(\Gamma)}{\eta(\Gamma)\leq M}}\inf_{\rho\in C_+(X)}\Lag(\rho,\eta)
    =\int_\Gamma b d\eta^*.
\end{align}
For properties (4) and (5), let $(\rho_k)$ be a minimizing sequence in $\A(\Gamma,b)\cap C(X)$ for $\Mod_p^c(\Gamma,b)$. Then, by Lemma \ref{lem:continuous-densities} and Lemma \ref{lem:primal-minimizer} we have $\rho_k\xrightarrow{k\to\infty}\rho^*$ in $L^p(X)$.  Then
\begin{align*}
    \inf_{\rho\in C_+(X)}\Lag(\rho,\eta^*)\leq \Lag(\rho_k,\eta^*)
    \leq\sup_{\stackrel{\eta\in\M_+(\Gamma)}{\eta(\Gamma)\leq M}}\Lag(\rho_k,\eta). 
\end{align*}
Since $\eta^*\in\B(\Gamma)$, Lemma \ref{lem:dual-penalty} implies that
$$ \inf_{\rho\in C_+(X)}\Lag(\rho,\eta^*)=\int_\Gamma bd\eta^*. $$
Similarly, since $\rho_k\in\A(\Gamma,b)$ Lemma \ref{lem:primal-penalty} implies that
$$ \sup_{\stackrel{\eta\in\M_+(\Gamma)}{\eta(\Gamma)\leq M}}\Lag(\rho_k,\eta)
\leq\sup_{\eta\in\M_+(\Gamma)}\Lag(\rho_k,\eta)=\|\rho_k\|_{L^p(X)}. $$
We conclude that
\begin{align*}
    \int_\Gamma bd\eta^*\leq \Lag(\rho_k,\eta^*)\leq \|\rho_k\|_{L^p(X)}.
\end{align*}
Letting $k\to\infty$ while employing \eqref{eq:strong-duality-in-proof} yields
\begin{align*}
    \int_\Gamma bd\eta^*=\lim_{k\to\infty}\Lag(\rho_k,\eta^*)=\|\rho^*\|_{L^p(X)}.
\end{align*}
Moreover,
\begin{align*}
    \Lag(\rho_k,\eta^*)
    &=\|\rho_k\|_{L^p(X)}-\la \rho_k,A^\intercal\eta^*\ra+\la b,\eta^*\ra \\
    &=\|\rho_k\|_{L^p(X)}-\int_X\rho_k\frac{d(A^\intercal\eta^*)}{d\mu}d\mu+\int_\Gamma bd\eta^* \\
    &\quad\xrightarrow{k\to\infty} \|\rho^*\|_{L^p(X)}-\int_X\rho^*\frac{d(A^\intercal\eta^*)}{d\mu}d\mu+\int_\Gamma bd\eta^*
    =\Lag(\rho^*,\eta^*).
\end{align*}
Now we have obtained that
\begin{align*}
    \int_\Gamma bd\eta^*=\Lag(\rho^*,\eta^*)=\|\rho^*\|_{L^p(X)}.
\end{align*}
In particular,
$$ \|\rho^*\|=\int_X\rho^*\frac{d(A^\intercal\eta^*)}{d\mu}d\mu
\quad\text{and}\quad \int_\Gamma \int_\gamma\rho^*ds-b(\gamma)d\eta^*=0.  $$
Properties (4) and (6) follow from this.
Property (7) follows from property (6) and formula \eqref{eq:transpose-formula-new}. 

Finally, for property (5), we combine properties (1) and (4) and Corollary \ref{cor:p-exceptional-is-eta-nul} as follows.
The exceptional set $E:=\{\gamma\in\Gamma:A\rho^*(\gamma)<b(\gamma)\}$ is Borel by Lemma \ref{lem:regularity-of-various-maps-on-Curv} and the continuity of $b$, and $\Mod_p(E)=0$ by property (1).
Corollary \ref{cor:p-exceptional-is-eta-nul} implies that $\eta(E)=0$ for all $\eta\in\B(E)$.
Note that $\eta^*\resmes E\in\B(E)$ by Remark \ref{rem:small-lemma}. Thus $\eta^*(E)=0$.
It follows that $A\rho_n\geq b$ holds $\eta^*$-a.e. and this combined with property (4) yields that $A\rho^*=b$ holds $\eta^*$-a.e., as desired.
\end{proof}
\section{Gradient curves for \texorpdfstring{$p$}{p}-harmonic functions} \label{sec:proof-of-main}

In this section we give the proof of the main theorem, Theorem \ref{thm:intro-main}.  
We first explain the idea of the proof. 
Let $X=(X,d,\mu)$ be a complete, doubling $p$-Poincar\'{e} space, $1<p<\infty$.
Let $\Om'\subset X$ be a domain and $u\in N^{1,p}_{\loc}(\Om')$ a $p$-harmonic function. 
We fix a bounded subdomain $\Om\subset\subset\Om'$ such that $\mu(\partial\Om)=0$ and $u$ is nonconstant in $\Om$. 
As explained in Section \ref{sec:energy}, $u|_\Om$ is the unique minimizer of 
\begin{equation}
    \inf\Big\{\int_{\Om} g_v^pd\mu:v-u|_\Om\in N^{1,p}_0(\Om)\Big\}.
\end{equation}
By Proposition \ref{prop:equivalence} we find that $g_u=\rho^*$ $\mu$-a.e. in $\Om$, where $\rho^*$ is the unique weak minimizer of the generalized modulus problem
\[ \Mod_p(\Gamma,b)=
\inf\Big\{\int_{\overline{\Om}}\rho d\mu:\rho\in \A(\Gamma,b)\Big\}. \]
Here
\begin{equation} \label{eq:Gamma}
    \Gamma:=\{\gamma\in\Curv(\overline{\Om}):\gamma(0),\gamma(\ell_\gamma)\in\partial\Om\}
\end{equation}
is the family of all boundary-to-boundary curves in $\overline{\Om}$ (which is closed by Lemma \ref{lem:regularity-of-various-maps-on-Curv}) and 
\begin{equation} \label{eq:b}
    b(\gamma)=|u(\gamma(0))-u(\gamma(\ell_\gamma))|\quad\text{for all }\gamma\in\Gamma.
\end{equation}
Since $u$ is nonconstant in $\Om$, we have $0<\Mod_p(\Gamma,b)<\infty$. Function $b$ is continuous on $\Gamma$ by the following lemma:

\begin{lemma}
Let $\Gamma$ denote the family of all boundary-to-boundary curves in $\overline{\Om}$, as in \eqref{eq:Gamma}, and $f\in C(\overline{\Om})$. Define $b\colon\Gamma\to\R$ by
$$ b(\gamma):=|f(\gamma(0))-f(\gamma(\ell_\gamma))|\quad\text{for all }\gamma\in\Gamma. $$
Then $b\in C(\Gamma)$.
\end{lemma}

\begin{proof}
The endpoint evaluation maps are continuous from $(\Gamma,d_{\Curv})$ to $\overline{\Om}$ by Lemma \ref{lem:regularity-of-various-maps-on-Curv}, $f$ is continuous from $\partial\Om$ to $\R$, and the summation and absolute value operations are continuous functions on $\R$.
\end{proof}

We would like to use the duality machinery of Section \ref{sec:generalization}, and in particular Proposition \ref{prop:properties}, to find a Radon measure $\eta^*$ on $\Gamma$ that selects the gradient curves of $u$.
However, there are two serious issues that prevent us from employing Proposition \ref{prop:properties}: the family $(\Gamma,d_{\Curv})$ is not a compact family of curves (as it contains arbitrarily long curves) and $\inf_\Gamma b=0$.

We circumvent the problem by an approximation argument. Let us note that $\overline{\Om}$ is a compact set. We write
\begin{equation} \label{eqq:decomp-of-Gamma-1}
    \Gamma=\Gamma_0\cup\Big(\bigcup_{n=1}^\infty\Gamma_n\Big)
\end{equation}
where
$$ \Gamma_0:=\{\gamma\in\Gamma:b(\gamma)=0\} 
\quad\text{and}\quad
\Gamma_n:=\{\gamma\in\Gamma:\ell(\gamma)\leq n\text{ and } b(\gamma)\geq \tfrac{1}{n}\}\quad\text{for }n=1,2,3,\ldots. $$
For a fixed $n=1,2,3,\ldots$, $\Gamma_n$ is a compact family of curves due to Lemma \ref{lem:curves-with-bdd-length-compact} and the continuity of $b$, and $\inf_{\Gamma_n}b \geq \frac{1}{n}$. Consequently, by Lemma \ref{lem:continuous-densities}
$$ \Mod_p(\Gamma_n,b)=\Mod_p^c(\Gamma_n,b). $$
Now we are in a position to apply Proposition \ref{prop:properties} with $\Gamma_n\subset\Curv(\overline{\Om})$ and $b\in C(\Gamma_n)$. We conclude the following.

\begin{proposition} \label{prop:approximate-rhos-and-etas}
Let $X=(X,d,\mu)$ be a complete, doubling $p$-Poincar\'{e} space, $1<p<\infty$.
Suppose that $\Om'\subset X$ is a domain and $u\in N^{1,p}_{\loc}(\Om')$ a nonconstant $p$-harmonic function. 
Let us fix a bounded subdomain $\Om\subset\subset\Om'$ such that $\mu(\partial\Om)=0$. Let
$$ \Gamma:=\{\gamma\in\Curv(\overline{\Om}):\gamma(0),\gamma(\ell_\gamma)\in\partial\Om\} $$ 
and 
$$ b(\gamma)=|u(\gamma(0))-u(\gamma(\ell_\gamma))|\quad\text{for all }\gamma\in\Gamma. $$
We decompose $\Gamma$ as in \eqref{eqq:decomp-of-Gamma-1}.
Then for any $n=1,2,\ldots$ there exist $\rho_n\in\LB^p_+(\overline{\Om})$ and $\eta_n\in\M_+(\Gamma_n)$ such that the following hold:
\begin{enumerate}
    \item $\int_\gamma\rho_nds\geq b(\gamma)$ for $p$-a.e. $\gamma\in\Gamma_n$.
    \item $\eta_n(\Gamma_n)\leq M_n:=n\big(\Mod_p^c(\Gamma_n,b)\big)^{1/p}$ and $\eta_n\in\B(\Gamma_n)$.
    \item $(\Mod_p^c(\Gamma_n,b))^{1/p}=\|\rho_n\|_{L^p(\overline{\Om})}=\int_{\Gamma_n}bd\eta_n$.
    \item We have that
    $$ \int_{\Gamma_n}\int_\gamma\rho_nds-b(\gamma)d\eta_n=0. $$ 
    \item Measure $\eta_n$ is supported on those curves $\gamma\in\Gamma_n$ for which $\int_\gamma\rho_nds=b(\gamma)$. In other words, $\int_\gamma\rho_nds=b(\gamma)$ holds for $\eta_n$-a.e. curve $\gamma\in\Gamma_n$.
    \item For $\mu$-a.e.
    $$ \frac{d(A^\intercal\eta_n)}{d\mu}=\big(\frac{\rho_n}{\|\rho_n\|_p}\big)^{p-1} $$
    \item For any $\rho\in \Bor_+(X)$ we have that
    \begin{equation} \label{eq:useful-identity-for-dual-measure}
    \int_\Gamma\int_\gamma\rho dsd\eta_n=\int_{\overline{\Om}}\rho\Big(\frac{\rho_n}{\|\rho_n\|_p}\Big)^{p-1}d\mu.    
    \end{equation}
\end{enumerate}    
\end{proposition}

The goal of this section is to prove that $\rho_n\to\rho^*$ where $\rho^*=g_u$, as explained above, and $\eta_n\to\eta^*$ for a suitable limit measure $\eta^*$, and that properties (1)-(7) are preserved when we pass to the limit $n\to\infty$.
In Section \ref{ssec:convergence-of-rhos} we consider the convergence $\rho_n\to\rho^*$ and prove that $\|\rho_n\|_{L^p(\overline{\Om})}\to\|\rho^*\|_{L^p(\overline{\Om})}$. In Section \ref{ssec:convergenge-of-etas} we consider the convergence $\eta\to\eta^*$ and prove that $\int_{\Gamma_n}bd\eta_n\to\int_\Gamma bd\eta^*$. In Section \ref{ssec:key-propeties-of-rho-and-eta} we check that properties (4)-(7) are preserved at the limit. In Section \ref{ssec:conclusion} we collect these facts and prove Theorem \ref{thm:intro-main}.
\subsection{Convergence of \texorpdfstring{$(\rho_n)$}{rhon}} \label{ssec:convergence-of-rhos}

\begin{lemma} \label{lem:convergence-of-rhon}
Under the assumptions of Proposition \ref{prop:approximate-rhos-and-etas}, let $\rho^*\in\LB^p_+(\overline{\Om})$ be the weak minimizer of $\Mod_p(\Gamma,b)$ and $\rho_n\in\LB^p_+(\overline{\Om})$ be the weak minimizer of $\Mod_p^c(\Gamma_n,b)$. Then $\rho_n\xrightarrow{n\to\infty}\rho^*$ in $L^p(\overline{\Om})$, up to a subsequence.
\end{lemma}

\begin{proof}
Note that
$$ \A(\Gamma_1,b)\supset\A(\Gamma_2,b)\supset
\ldots\supset\A(\Gamma,b), $$
and thus
\begin{equation} \label{eq:monotonicity-of-rhon}
    \|\rho_1\|_{L^p(\overline{\Om})}\leq\|\rho_2\|_{L^p(\overline{\Om})}\leq \ldots\leq\|\rho^*\|_{L^p(\overline{\Om})}.
\end{equation}
It follows that $(\rho_n)_{n=1}^\infty$ is bounded in $L^p(\overline{\Om})$ and therefore there exists $\Tilde{\rho}\in\LB^p(\overline{\Om})$ such that
$\rho_n\xrightharpoonup{n\to\infty}\Tilde{\rho}$ weakly in $L^p(\overline{\Om})$, up to a subsequence which we still denote by itself. By a standard ``subsequence of a subsequence argument'' it will suffice to prove that, along all such subsequences where the weak limit exists, $\Tilde{\rho}=\rho^*$ a.e. and that $\rho_n \to \Tilde{\rho}$ in $L^p(X)$. 

Each $\A(\Gamma_n,b)$ is a convex subset of $L^p(\overline{\Om})$, and thus by Mazur's and Fuglede's lemmas we conclude that for any fixed $n\in\N$ we have $\tilde{\rho}\in\overline{\A(\Gamma_n,b)}$, i.e.
\[ \int_\gamma\tilde{\rho}ds\geq b(\gamma)
\quad\text{for $p$-a.e. }\gamma\in\Gamma_n. \]
We use Lemma \ref{lem:relaxation-of-admissibility} and Lemma \ref{lem:continuous-densities} to conclude $\|\rho_n\|_{L^p(\overline{\Om})}=\Mod_p^c(\Gamma_n,b)=\Mod_p(\Gamma_n,b)\leq \|\tilde{\rho}\|_{L^p(\overline{\Om})}$.
It follows that $\limsup_{n\to\infty}\|\rho_n\|_{L^p(\overline{\Om})}\leq\|\tilde{\rho}\|_{L^p(\overline{\Om})}$, and thus we can update the weak convergence $\rho_n\xrightharpoonup{n\to\infty}\Tilde{\rho}$ into strong convergence $\rho_n\xrightarrow{n\to\infty}\Tilde{\rho}$.

Now we have obtained the following properties of $\tilde{\rho}$:
$$ \rho_n\xrightarrow{n\to\infty}\tilde{\rho} 
\quad\text{and}\quad
\Tilde{\rho}\in\bigcap_{n=1}^\infty\overline{\A(\Gamma_n,b)}. $$
It is sufficient to prove that $\Tilde{\rho}$ is weakly admissible for $\Gamma$. Note that $\|\Tilde{\rho}\|_{L^p(\overline{\Omega})}=\lim_{n\to \infty} \|\rho_n\|_{L^p(\overline{\Omega})}\leq \Mod_p(\Gamma,b)$. Then, by the uniqueness of the weak minimizer and Lemma \ref{lem:primal-minimizer}  we have $\tilde{\rho}=\rho^*$ $\mu$-a.e. in $\overline{\Om}$. 

We check that the weak admissibility, i.e. that $\tilde{\rho}\in\overline{\A(\Gamma,b)}$ holds. For a fixed $n\in\N$, since $\rho\in\overline{\A(\Gamma,n)}$, there exists an exceptional set of curves $E_n\subset\Gamma_n$ such that
$\Mod_p(E_n)=0$ and 
$$ \int_\gamma\tilde{\rho}ds\ge b(\gamma)
\quad\text{for all }\gamma\in \Gamma_n\setminus E_n. $$
By \cite{Bjorn2011}*{Corollary 1.38} we have $\Mod_p(\cup_{n=1}^\infty E_n)=0$. 
Recall that $\Gamma=\Gamma_0\cup(\bigcup_{n=1}^\infty \Gamma_n)$. We claim that
$$ \int_\gamma\tilde{\rho}ds\ge b(\gamma)
\quad\text{for all }\gamma\in \Gamma\setminus 
\Big(\bigcup_{n=1}^\infty E_n\Big). $$
Indeed, if $\gamma\in\Gamma_0$, then
$$ \int_\gamma\tilde{\rho}ds\geq 0=b(\gamma), $$
and
$$ \int_\gamma\tilde{\rho}ds\ge b(\gamma)
\quad\text{for all }\gamma\in \Big(\bigcup_{n=1}^\infty\Gamma_n\Big)\setminus 
\Big(\bigcup_{n=1}^\infty E_n\Big). $$
The proof is finished by applying Lemma \ref{lem:closure}.
\end{proof}

The previous lemma together with Proposition \ref{prop:equivalence} implies the following.

\begin{proposition} \label{prop:modulus-at-the-limit}
Under the assumptions of Proposition \ref{prop:approximate-rhos-and-etas}, let $\rho^*\in\LB^p_+(\overline{\Om})$ be the weak minimizer of $\Mod_p(\Gamma,b)$ and $\rho_n\in\LB^p_+(\overline{\Om})$ be the weak minimizer of $\Mod_p^c(\Gamma_n,b)$.
We have the following chain of equalities:
$$ \int_\Om g_u^pd\mu=\int_{\overline{\Om}}(\rho^*)^pd\mu=\Mod_p(\Gamma,b)=\lim_{n\to\infty}\Mod_p^c(\Gamma_n,b)=\lim_{n\to\infty}\int_{\overline{\Om}}\rho_n^pd\mu. $$
\end{proposition}
\subsection{Convergence of \texorpdfstring{$(\eta_n)$}{etan}} \label{ssec:convergenge-of-etas}

For the convergence of $(\eta_n)$, let us point out that we cannot directly apply Prokhorov's theorem \ref{thm:prokhorov} to find a limit, because the $\eta_n$'s are not necessarily uniformly bounded. See Example \ref{ex:non-Radonness}. To fix this, we extract a limit of $(\eta_n)$ in two steps. In the first step, we extract ``local'' limits of $(\eta_n)$, and in the second step we build a desired global limit $\eta^*$ by using exhaustion.

For the local limits, let
\begin{equation} \label{eq:Tilde-Gamma-def}
    \tilde{\Gamma}_m:=\Big\{\gamma\in\Gamma:\frac{1}{m}\leq\operatorname{diam}(\gamma)\text{ and }\ell(\gamma)\leq m\Big\}
\end{equation}
for $m=1,2,\ldots$ so that
$$ \Gamma=\Gamma_0\cup\big(\bigcup_{m=1}^\infty\Tilde{\Gamma}_m\big). $$
By Lemma \ref{lem:regularity-of-various-maps-on-Curv} and Lemma \ref{lem:diam-continuous} each $\tilde{\Gamma}_m$ is compact.

The following lemma yields the existence of a subsequence of $(\eta_n)_{n=1}^\infty$, which we still denote by itself, such that the sequence of restricted measures $(\eta_n\resmes\tilde{\Gamma}_m)_{n=1}^\infty$ converges in the weak topology of $\M_+(\tilde{\Gamma}_m)$ for all $m=1,2,\ldots$.

\begin{lemma} \label{lem:local-limits}
Under the assumptions of Proposition \ref{prop:approximate-rhos-and-etas}, let $\eta_n\in\M_+(\Gamma_n)$ be the maximizer of $\eta\mapsto\int_{\Gamma_n}bd\eta$.
There exists a subsequence of $(\eta_n)_{n=1}^\infty$, still denoted by itself, such that for all $m=1,2,\ldots$ the sequence of restricted measures
$(\eta_n\resmes\tilde{\Gamma}_m)_{n=1}^\infty$ converges weakly to a Radon measure $\tilde{\eta}_m$ on $\tilde{\Gamma}_m$.
\end{lemma}

\begin{proof}
Let us fix $m=1,2,\ldots$ and consider the sequence of $\eta_n$'s, restricted to $\tilde{\Gamma}_m$, that is, $(\eta_n\resmes\tilde{\Gamma}_m)_{n=1}^\infty$.
Note that $\tilde{\Gamma}_m$ is compact by Lemma \ref{lem:regularity-of-various-maps-on-Curv} and  Lemma \ref{lem:diam-continuous}. Moreover, $(\eta_n\resmes\tilde{\Gamma}_m)_{n=1}^\infty$ is uniformly bounded. Indeed, if we use \eqref{eq:useful-identity-for-dual-measure} with $\rho=1$, we obtain that
$$ \int_{\Gamma_n}\ell(\gamma)d\eta_n=\int_{\overline{\Om}}\Big(\frac{\rho_n}{\|\rho_n\|_{L^p(\overline{\Om})}}\Big)^{p-1}d\mu. $$
We estimate the right-hand side of the above inequality from above by Hölder's inequality:
\begin{align*}
    \int_{\overline{\Om}}\Big(\frac{\rho_n}{\|\rho_n\|_{L^p(\overline{\Om})}}\Big)^{p-1}d\mu
    \leq \mu(\overline{\Om})^{1/p}.
\end{align*}
We estimate the left-hand side from below as follows:
\begin{align*}
    \int_{\Gamma_n}\ell(\gamma)d\eta_n
    \geq
    \int_{\Gamma_n\cap\Tilde{\Gamma}_m}\ell(\gamma)d\eta_n
    \geq
    \int_{\Gamma_n\cap\tilde{\Gamma}_m}\operatorname{diam}(\gamma)d\eta_n
    \geq
    \frac{1}{m}\eta_n(\Gamma_n\cap\tilde{\Gamma}_m)
    =\frac{1}{m}(\eta_n\resmes\tilde{\Gamma}_m)(\tilde{\Gamma}_m).
\end{align*}
Here we used the fact that $\operatorname{diam}(\gamma)\leq\ell(\gamma)$.
It follows that
$$ (\eta_n\resmes\tilde{\Gamma}_m)(\tilde{\Gamma})\leq m\mu(\overline{\Om})^{1/p}
\quad\text{for all }n=1,2,3,\ldots. $$
By Prokhorov's theorem, Theorem \ref{thm:prokhorov}, $(\eta_n\resmes\tilde{\Gamma}_m)_{n=1}^\infty$ has a subsequential limit $\tilde{\eta}_m$ with respect to the weak topology of $\M_+(\tilde{\Gamma}_m)$. 

Finally, we may use the standard diagonal argument to extract a single subsequence of $(\eta_n)_{n=1}^\infty$, still denoted by itself, such that for all $m=1,2,\ldots$ 
$$ \eta_n\resmes\tilde{\Gamma}_m\xrightarrow{n\to\infty}\tilde{\eta}_m $$ 
with respect to the weak topology of $\M_+(\tilde{\Gamma}_m)$. 
\end{proof}

\begin{lemma}[Properties of local limits] \label{lem:properties-of-local-limits}
Under the assumptions of Proposition \ref{prop:approximate-rhos-and-etas}, let $\Gamma_n$ and $\eta_n$ be given by Proposition \ref{prop:approximate-rhos-and-etas}, $\Tilde{\Gamma}_m$ as in \eqref{eq:Tilde-Gamma-def}, and let $(\eta_n)_{n=1}^\infty$ denote a sequence given by Lemma \ref{lem:local-limits} such that for all $m=1,2,\ldots$ 
$$ \eta_n\resmes\tilde{\Gamma}_m\xrightarrow{n\to\infty}\tilde{\eta}_m $$ 
with respect to the weak topology of $\M_+(\tilde{\Gamma}_m)$.
Then the measures $(\tilde{\eta}_m)_{m=1}^\infty$ have the following properties:
\begin{enumerate} 
    \item Monotonicity: $\tilde{\eta}_m(\Sigma )\leq\tilde{\eta}_{m+1}(\Sigma )$ for all $\Sigma \subset\Gamma$ Borel.
    \item $\int_\Gamma bd\tilde{\eta}_m\leq (\Mod_p(\Gamma,b))^{1/p}$.
    \item $\tilde{\eta}_m\in\B(\Tilde{\Gamma}_m)$.
\end{enumerate}
\end{lemma}

\begin{proof}
\begin{enumerate}
    \item Let $\Sigma \subset\Gamma$ be open. Then $\mathbbm{1}_\Sigma $ is lower semi-continuous and we can find a sequence $(h_i)$ of continuous, nonnegative functions such that $h_i\nearrow\mathbbm{1}_\Sigma $. Then
    $$ \int_\Gamma h_id(\eta_n\resmes\tilde{\Gamma}_m)\leq \int_\Gamma h_id(\eta_n\resmes\tilde{\Gamma}_{m+1}). $$
    Since $h_i\in C(\Gamma)$ is bounded, we may let $n\to\infty$ to obtain
    $$ \int_\Gamma h_id\tilde{\eta}_m\leq \int_\Gamma h_id\tilde{\eta}_{m+1}. $$
    Since $h_i\nearrow\mathbbm{1}_\Sigma $ we may let $i\to\infty$ to obtain
    $$ \tilde{\eta}_m(\Sigma )\leq \tilde{\eta}_{m+1}(\Sigma ). $$
    Finally, since $\tilde{\eta}_m$ and $\tilde{\eta}_{m+1}$ are Radon, we conclude that the claim holds for all Borel sets $\Sigma \subset\Gamma$.
    \item Since $b\in C(\Gamma)$ is bounded,
    $$ \int_\Gamma bd\tilde{\eta}_m=\lim_{n\to\infty}\int_\Gamma bd(\eta_n\resmes\tilde{\Gamma}_m). $$
    We use Proposition \ref{prop:approximate-rhos-and-etas}, item (3), and Proposition \ref{prop:modulus-at-the-limit} to obtain
    $$ \int_\Gamma bd(\eta_n\resmes\tilde{\Gamma}_m) 
    \leq \int_\Gamma bd\eta_n=(\Mod_p(\Gamma_n,b))^{1/p}\xrightarrow{n\to\infty}(\Mod_p(\Gamma,b))^{1/p}. $$
    It follows that $\int_\Gamma b d\tilde{\eta}_m\leq(\Mod_p(\Gamma,b))^{1/p}$.
    \item By Proposition \ref{prop:approximate-rhos-and-etas}, item (2), we know that $\eta_n\in\B(\Gamma_n)$. By Remark \ref{rem:small-lemma} we have that $\eta_n\resmes\tilde{\Gamma}_m\in\B(\Tilde{\Gamma}_m)$. Since $\eta_n\resmes\tilde{\Gamma}_m\xrightarrow{n\to\infty}\tilde{\eta}_m$ weakly in $\M_+(\tilde{\Gamma}_m)$, Lemma \ref{lem:B-closed-wrt-weak-convergence-to-Radon-measure} implies that $\tilde{\eta}_m\in\B(\tilde{\Gamma}_m)$.
\end{enumerate}
\end{proof}

We are now ready to proceed to the second step: construction of a global limit of $(\eta_n)_{n=1}^\infty$. By Lemma \ref{lem:properties-of-local-limits}, item (1), we can define
\begin{equation} \label{eq:def-of-etastar}
    \eta^*(\Sigma ):=\lim_{m\to\infty}\tilde{\eta}_m(\Sigma )
    \quad\text{for all }\Sigma \subset\Gamma\text{ Borel.}
\end{equation}
It follows that $\eta^*$ is a well-defined Borel measure on $\Gamma$. It defines a unique complete measure denoted by the same letter. If $h\in\Bor_+(\Gamma)$, then by the definition of integral it follows that
\begin{equation} \label{eq:integral-wrt-etastar}
    \int_\Gamma hd\eta^*=\lim_{m\to\infty}\int_\Gamma hd\tilde{\eta}_m.
\end{equation}

One might expect $\eta^*\in\B(\Gamma)$, but it turns out that $\eta^*$ does need to be Radon. See Example \ref{ex:non-Radonness} below. However, the following holds:

\begin{lemma} \label{lem:etastar-in-B}
Under the assumptions of Proposition \ref{prop:approximate-rhos-and-etas}, let $\eta^*$ be defined by \eqref{eq:def-of-etastar}. Then $A^\intercal\eta^*$ is absolutely continuous with respect to $\mu$ and $\|\frac{d(A^\intercal\eta^*)}{d\mu}\|_{L^q(\overline{\Om})}\leq 1$.
\end{lemma}

\begin{proof}
By the convergence \eqref{eq:integral-wrt-etastar} and Lemma \ref{lem:properties-of-local-limits}, we may apply Lemma \ref{lem:B-weakly-closed-and-double-limit-in-dual-pairing}, item (1), to obtain the claim.
\end{proof}

The following proposition is a counterpart to Proposition \ref{prop:modulus-at-the-limit}. It says that at the limit $n\to\infty$ we achieve "strong duality".

\begin{proposition} \label{prop:duality-at-the-limit}
Under the assumptions of Proposition \ref{prop:approximate-rhos-and-etas}, let $\eta^*$ be defined by \eqref{eq:def-of-etastar}. Then
$$ \int_\Gamma bd\eta^*=(\Mod_p(\Gamma,b))^{1/p}. $$
\end{proposition}

\begin{proof}
By Lemma \ref{lem:properties-of-local-limits}, item (2), and \eqref{eq:integral-wrt-etastar}, we can deduce that
$$ \int_\Gamma bd\eta^*\leq (\Mod_p(\Gamma,b))^{1/p}. $$
In particular $\int_\Gamma bd\eta^*<\infty$.

For the converse inequality, we claim that for any $\epsilon>0$ there exists $M\in\N$ such that $m\geq M$ implies that
$$ (\Mod_p(\Gamma,b))^{1/p}-\epsilon\leq\int_\Gamma bd\tilde{\eta}_m. $$
Let us fix $\epsilon>0$.
Since $b\in C(\Gamma)$ is bounded, by the definition of $\tilde{\eta}_m$ as a weak limit,
\begin{align*}
    \int_\Gamma bd\tilde{\eta}_m
    =\lim_{n\to\infty}\int_\Gamma bd(\eta_n\resmes\tilde{\Gamma}_m)
    =\lim_{n\to\infty}\Big(\int_\Gamma bd\eta_n-\int_{\Gamma\setminus\tilde{\Gamma}_m}bd\eta_n\Big).
\end{align*}
In the last equality we used the fact that the integrals are finite.

By Proposition \ref{prop:approximate-rhos-and-etas}, item (3), and Proposition \ref{prop:modulus-at-the-limit}
$$ \lim_{n\to\infty}\int_\Gamma bd\eta_n
=\lim_{n\to\infty}\|\rho_n\|_{L^p(\overline{\Om})}=\Mod_p(\Gamma,b)^{1/p}. $$
On the other hand,
\begin{align*}
    \int_{\Gamma\setminus\tilde{\Gamma}_m}bd\eta_n
    \leq \int_{\Sigma_m^1}bd\eta_n
    +\int_{\Sigma_m^2}bd\eta_n
\end{align*}
where
$$ \Sigma_m^1:=\{\gamma\in\Gamma:\operatorname{diam}(\gamma)<\tfrac{1}{m}\}
\quad\text{and}\quad
\Sigma_m^2:=\{\gamma\in\Gamma:\ell(\gamma)>m\}. $$
We employ the weak duality, Lemma \ref{lem:weak-duality}, for the generalized modulus problems
$$ \Mod_p(\Sigma_m^1,b)\quad\text{and}\quad
\Mod_p(\Sigma_m^2,b) $$
to obtain that
\begin{align} \label{eq:Gamma-minus-Gammam-estimate}
    \int_{\Gamma\setminus\tilde{\Gamma}_m}bd\eta_n
    \leq \Mod_p(\Sigma_m^1,b)^{1/p}
    +\Mod_p(\Sigma_m^2,b)^{1/p}.
\end{align} 
(Technically speaking, here Lemma \ref{lem:weak-duality} is applied to the measure $\eta_n'\resmes(\Gamma\setminus\tilde{\Gamma}_m)\in\B(\Gamma\setminus\tilde{\Gamma}_m)$ where $\eta_n'$ denotes the zero extension of $\eta_n$ from $\Gamma_n$ to $\Gamma$.)
We estimate the two items on the right hand side of the above inequality separately. Choose $h\in \mathcal{A}(\Gamma,b)$ with $\|h\|_{L^p}<\infty$.
For the first item, let 
\[ \overline{\Omega}_m:=\{x\in\overline{\Omega}:\operatorname{dist}(x,\partial\Om)\leq \tfrac{1}{m}\}. \]
Then all the curves in $\Sigma_m^1$ lie inside $\overline{\Om}_m$, and thus $h\mathbbm{1}_{\overline{\Om}_m}\in\A(\Sigma_m^1,b)$. 
Then
\begin{align} \label{eq:Sigma1-estimate}
    \Mod_p(\Sigma_m^1,b)
    \leq
    \int_{\overline{\Om}_m}h^pd\mu\xrightarrow{m\to\infty}0. 
\end{align}
For the second item, we employ Lemma \ref{lem:Mod-inequality} and the fact that $\frac{1}{m}\mathbbm{1}_{\overline{\Om}}\in\A(\Sigma_m^2,1)$ to obtain
\begin{align} \label{eq:Sigma2-estimate}
    \Mod_p(\Sigma_m^2,b)
    \leq \|b\|_\infty^p\Mod_p(\Sigma_m^2,1)
    \leq \tfrac{1}{m^p}\|b\|_\infty^p\mu(\overline{\Om})^{1/p}.
\end{align}
We combine estimates \eqref{eq:Gamma-minus-Gammam-estimate}, \eqref{eq:Sigma1-estimate} and \eqref{eq:Sigma2-estimate} to find $M\in\N$, independent of $n$, such that $m\geq M$ implies that
$$ \int_{\Gamma\setminus\tilde{\Gamma}_m}bd\eta_n<\epsilon. $$
We conclude that
\begin{align*}
    \int_\Gamma bd\tilde{\eta}_m
    =\lim_{n\to\infty}\Big(\int_\Gamma bd\eta_n-\int_{\Gamma\setminus\tilde{\Gamma}_m}bd\eta_n\Big)
    \geq 
    \lim_{n\to\infty}\int_\Gamma bd\eta_n-\epsilon
    =\Mod_p(\Gamma,b)^{1/p}-\epsilon.
\end{align*}
We let $m\to\infty$ and the proof is finished.
\end{proof}
\subsection{Key properties of \texorpdfstring{$\rho^*$}{rhoast} and \texorpdfstring{$\eta^*$}{etaast}} \label{ssec:key-propeties-of-rho-and-eta}
This section is devoted to the checking that the joint properties of the functions $\rho_n$ and the measures $\eta_n$ in Proposition \ref{prop:approximate-rhos-and-etas} (items (4)-(7)) are preserved at the limit $n\to\infty$.

\begin{proposition} \label{prop:etastar-picks-gradient-curves}
Under the assumptions of Proposition \ref{prop:approximate-rhos-and-etas}, let $\rho^*\in\LB^p_+(\overline{\Om})$ be the weak minimizer of $\Mod_p(\Gamma,b)$ and let $\eta^*$ be defined by \eqref{eq:def-of-etastar}. Then
\begin{equation} \label{eq:estar-picks-gradient-curves}
    \int_\Gamma\int_\gamma\rho^*ds-b(\gamma)d\eta^*=0.
\end{equation}
Moreover, $\eta^*$ is supported on those curves $\gamma\in\Gamma$ for which $\int_\gamma\rho ds=b(\gamma)$.
\end{proposition}

\begin{proof}
Since $\rho^*$ is a weak minimizer, we know that $\int_\gamma\rho^*ds\geq b(\gamma)$ for all $\gamma\in\Gamma\setminus E$, where $E\subset\Gamma$ is such that $\Mod_p(E)=0$. 
Note that we can choose $E$ to be Borel by Lemma \ref{lem:regularity-of-various-maps-on-Curv} and the continuity of $b$.
It follows that $\eta(E)=0$ for all $\eta\in\B(E)$ by Corollary \ref{cor:p-exceptional-is-eta-nul}.
By Remark \ref{rem:small-lemma} $\tilde{\eta}_m\resmes E\in \B(E)$, and thus $\tilde{\eta}_m(E)=0$ for all $m=1,2,\ldots$. It follows that $\eta^*(E)=0$ and thus
\begin{equation} \label{eq:estar-picks-gradient-curves-inequality}
    \int_\Gamma\int_\gamma\rho^*ds-b(\gamma)d\eta^*
    =\int_{\Gamma\setminus E}\int_\gamma\rho^*ds-b(\gamma)d\eta^*\geq0.
\end{equation}

For the converse inequality, by Lemma \ref{lem:local-limits} we know that $(\eta_n\resmes\tilde{\Gamma}_m)_{n=1}^\infty$ converges weakly to $\tilde{\eta}_m$ for all $m\in\N$
and $$ \Big\|\frac{d(A^\intercal(\eta_n\resmes\tilde{\Gamma}_m))}{d\mu}\Big\|_{L^q(\overline{\Om}}
\leq\Big\|\frac{d(A^\intercal\eta_n)}{d\mu}\Big\|_{L^q(\overline{\Om}}\leq 1 $$
for all $m\in\N$. (See Remark \ref{rem:restriction-and-AT}.)
Moreover, $\rho_n\xrightarrow{n\to\infty}\rho^*$ in $L^p(\overline{\Om})$, by Lemma \ref{lem:convergence-of-rhon}. 
Thus we can employ Lemma \ref{lem:B-weakly-closed-and-double-limit-in-dual-pairing}, item (2), to obtain that
\begin{align*}
    \la A\rho^*,\tilde{\eta}_m\ra
    \leq \liminf_{n\to\infty}\la A\rho_n,\eta_n\resmes\tilde{\Gamma}_m
    \ra
\end{align*}
holds for all $m\in\N$. 
By Proposition \ref{prop:approximate-rhos-and-etas}, item (5), and the weak convergence of $(\eta_n\resmes\tilde{\Gamma}_m)_{n=1}^\infty$ we conclude that
\begin{align*}
    \la A\rho^*,\tilde{\eta}_m\ra
    \leq \liminf_{n\to\infty}\la A\rho_n,\eta_n\resmes\tilde{\Gamma}_m\ra
    = \liminf_{n\to\infty}\la b,\eta_n\resmes\tilde{\Gamma}_m\ra 
    =\la b,\tilde{\eta}_m\ra
\end{align*}
for all $m\in\N$. Now we may let $m\to\infty$ to obtain
$$ \la A\rho^*,\eta^*\ra\leq \la b,\eta^*\ra. $$
This is the converse inequality to \eqref{eq:estar-picks-gradient-curves-inequality}.

Finally, since $\int_\gamma\rho^*ds\geq b(\gamma)$ for $\eta^*$-a.e. $\gamma\in\Gamma$ by the above argument we can use \eqref{eq:estar-picks-gradient-curves} to conclude that $\int_\gamma\rho^*ds=b(\gamma)$ for $\eta^*$-a.e. $\gamma\in\Gamma$.
\end{proof}

\begin{proposition} \label{prop:density-of-transpose-of-A}
Under the same assumptions as in Proposition \ref{prop:approximate-rhos-and-etas}, let $\rho^*\in\LB^p_+(\overline{\Om})$ be the weak minimizer of $\Mod_p(\Gamma,b)$ and let $\eta^*$ be defined by \eqref{eq:def-of-etastar}. Then $A^\intercal\eta^*$ is absolutely continuous with respect to $\mu$ and
\begin{equation}
    \frac{d(A^\intercal\eta^*)}{d\mu}=\Big(\frac{\rho^*}{\|\rho^*\|_{L^p(\overline{\Om})}}\Big)^{p-1}.
\end{equation}
In particular, for every $\rho\in\Bor_+(\overline{\Om})$
\begin{equation} \label{eq:integration-formula}
    \int_\Gamma\int_\gamma\rho dsd\eta^*=\int_{\overline{\Om}}\rho\Big(\frac{\rho^*}{\|\rho^*\|_{L^p(\overline{\Om})}}\Big)^{p-1}d\mu.
\end{equation}
\end{proposition}

\begin{proof}
We employ Proposition \ref{prop:modulus-at-the-limit}, Proposition \ref{prop:duality-at-the-limit}, Proposition \ref{prop:etastar-picks-gradient-curves}, \eqref{eq:transpose-formula-new} and Lemma \ref{lem:etastar-in-B} to obtain
\begin{align*}
    \|\rho^*\|_{L^p(\overline{\Om})}=(\Mod_p(\Gamma,b))^{1/p}
    =\int_\Gamma bd\eta^*
    =\int_\Gamma A\rho^*d\eta^*
    =\int_{\overline{\Om}}\rho^*d(A^\intercal\eta^*)
    =\int_{\overline{\Om}}\rho^*\frac{d(A^\intercal\eta^*)}{d\mu}d\mu
    \leq \|\rho^*\|_{L^p(\overline{\Om})}.
\end{align*}
In the rightmost inequality we employed Hölder's inequality and Lemma \ref{lem:etastar-in-B}.
We conclude that we have equality in the Hölder's inequality, from which the claim follows.
\end{proof}

Denote by length of a curve $\gamma$ within a Borel set $E$ by $\ell(\gamma|_E)=\int_\gamma \mathbbm{1}_E ds.$

\begin{corollary} \label{cor:gradient-curves-ae}
Under the assumptions of Proposition \ref{prop:approximate-rhos-and-etas}, let $E\subset\overline{\Om}$ be Borel such that $\mu(E)>0$ and $g_u>0$ in $E$. Then there exists a gradient curve $\gamma$ of $u$ that genuinely passes through $E$, namely $\ell(\gamma|_E)>0$.
\end{corollary}

\begin{proof}
Arguing by contradiction, suppose that there exists $E\subset\overline{\Om}$ such that $\mu(E)>0$ and no gradient curve passes through $E$.
We set $\rho=\mathbbm{1}_E$ in \eqref{eq:integration-formula} to obtain that 
$$ \int_\Gamma\ell(\gamma|_E)d\eta^*=\int_E\Big(\frac{\rho^*}{\|\rho^*\|_{L^p(\overline{\Om})}}\Big)^{p-1}>0. $$
It follows that there exists $\Sigma\subset\Gamma$ such that $\eta^*(\Sigma)>0$ and $\ell(\gamma|_E)>0$ for all $\gamma\in\Sigma$.
\end{proof}
\subsection{Conclusion} \label{ssec:conclusion}

In this section we give the proof of the main results, Theorem \ref{thm:intro-main} and Corollary \ref{cor:gradient-curves-ae}. For the reader's convenience, we restate these here:

\begin{theorem}
Let $X=(X,d,\mu)$ be a complete, doubling $p$-Poincar\'{e} space, $1<p<\infty$.
Suppose that $\Om'\subset X$ is a domain and $u\in N^{1,p}_{\loc}(\Om')$ a $p$-harmonic function. 
Let us fix a bounded subdomain $\Om\subset\subset\Om'$ such that $\mu(\partial\Om)=0$ and $u$ is nonconstant in $\Om$. 
There exists a Borel measure $\eta^*$ on
\begin{equation} \label{eq:conclusion-Gamma}
    \Gamma:=\{\gamma\in\Curv(\overline{\Om}):\gamma(0),\gamma(\ell_\gamma)\in\partial\Om\}
\end{equation}
such that the following hold:
\begin{enumerate}
    \item The $p$-energy of $u$ is given by integration over the gradient curves of $u$ with respect to $\eta^*$, 
    $$\|g_u\|_{L^p(\Om)}=\int_\Gamma|u(\gamma(0))-u(\gamma(\ell_\gamma))|d\eta^*. $$
    \item Measure $\eta^*$ is supported on the gradient curves on $u$, that is, on those curves $\gamma\in\Gamma$ for which
    $$ |u(\gamma(0))-u(\gamma(\ell_\gamma))|=\int_\gamma g_uds. $$
    \item For any Borel set $E\subset\overline{\Om}$, the following identity holds:
    \begin{equation} \label{eq:conclusion-identity}
        \int_\Gamma\int_\gamma\mathbbm{1}_Edsd\eta^*=\int_{E}\Big(\frac{g_u}{\|g_u\|_{L^p(\Om)}}\Big)^{p-1}d\mu.
    \end{equation}
    In particular, for $\mu$-a.e. $x\in\{y\in\Om:g_u(y)\neq 0\}$ there exists a gradient curve $\gamma$ of $u$ that passes through $x$.
\end{enumerate}
\end{theorem}

\begin{proof}
Let us define $\eta^*$ by \eqref{eq:def-of-etastar}. Then Proposition \ref{prop:modulus-at-the-limit} and Proposition \ref{prop:duality-at-the-limit} yield that
$$ \|g_u\|_{L^p(\Om)}=(\Mod_p(\Gamma,b))^{1/p}=\int_\Gamma|u(\gamma(0))-u(\gamma(\ell_\gamma))|d\eta^*. $$
Proposition \ref{prop:etastar-picks-gradient-curves} gives that 
$$ \int_\gamma g_uds=|u(\gamma(0))-u(\gamma(\ell_\gamma))| $$
for $\eta^*$-a.e. curve $\gamma\in\Gamma$. That is, by Remark \ref{rmk:gradientcurve} $\eta^*$ is supported on gradient curves on $u$. 

Identity \eqref{eq:conclusion-identity} follows from Proposition \ref{prop:density-of-transpose-of-A}. Finally, for a later contradiction, suppose that there exists a Borel set $E\subset\{y\in\Om:g_u(y)\neq 0\}$ such that $\mu(E)>0$ and for all $x\in E$ no gradient curve passes through $x$.
Then Corollary \ref{cor:gradient-curves-ae} gives the existence of a gradient curve $\gamma$ of $u$ such that $\ell(\gamma|_E)>0$. In particular there exists a point $x\in E$ such that $\gamma(t)=x$ for some $t\in[0,\ell_\gamma]$, a contradiction.
\end{proof}

\begin{figure}
\begin{tikzpicture}[scale=5]
    \draw[thick] (0,0) rectangle (1,1);
    \foreach \n in {1,2,3,4,5,6} {
        \pgfmathsetmacro{\lefts}{1/(2^\n)-1/(2^(\n+2))}
        \pgfmathsetmacro{\rights}{1/(2^\n)}
        \draw[fill=black!100, draw=none] (\lefts, 0) rectangle (\rights, 0.5);
    } 
    \node[below] at (0,0) {$0$};
    \node[below] at (1,0) {$1$};
    \node[left] at (0,1) {$1$};
    \node[left] at (0,0.5) {$\frac{1}{2}$};
\end{tikzpicture}
\caption{The domain $\Omega$ that is used in Example \ref{ex:non-Radonness}.}
\label{fig:nonradon}
\end{figure}
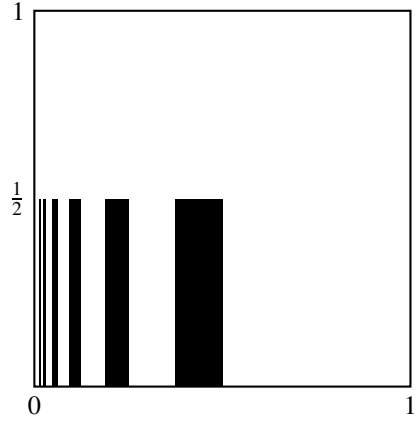

\begin{example} \label{ex:non-Radonness}
A subtle detail in Theorem \ref{thm:intro-main} is that the measure $\eta^*$ need not be finite. We equip $\R^2$ with the standard Euclidean metric and Lebesgue measure $\lambda$. 
Consider the domain 
\[\Omega=(0,1)^2\setminus \left(\bigcup_{n=1}^\infty \left[\frac{1}{2^n}-\frac{1}{2^{n+2}},\frac{1}{2^n}\right]\times\left[0,\frac{1}{2}\right]\right) \]
that is shown in Figure \ref{fig:nonradon} and the $p$-harmonic function $u\in N^{1,p}_{\rm loc}(\R^2)$ given by $u(x,y)=x$. The measure $\eta^*$ is supported on horizontal curves in $\Omega$. However, the curves must be contained within $\overline{\Omega}$ and for every $y\in (0,\frac{1}{2})$ the horizontal line $\R\times \{y\}$ intersects $\Omega$ in infinitely many components. 
We check that the measures of the horizontal curves restricted to these components add up to infinity.

Let $E_1=(\frac{1}{2},1)\times (0,\frac{1}{2})\subset\Om$ be the rightemost component of $\Omega \cap (0,1)\times(0,1/2)$ and let $\Gamma_1$ denote the horizontal curves in $E_1$. Theorem \ref{thm:intro-main}, item (3), implies that
\begin{equation} \label{eq:non-Radonness-formula}
    \int_\Gamma\ell(\gamma|_{E_1})d\eta^*=\lambda(E_1)\lambda(\Om)^{\frac{1-p}{p}}.
\end{equation}
Since $\int_\Gamma\ell(\gamma|_{E_1})d\eta^*=\frac{1}{2}\eta^*(\Gamma_1)$ and $\lambda(E_1)=\frac{1}{4}$, we conclude that $\eta^*(\Gamma_1)=\frac{1}{2}\lambda(\Om)^{\frac{1-p}{p}}$.

Similarly, for $n=2,3,\ldots$ let $E_n=(\frac{1}{2^n},\frac{1}{2^{n-1}}-\frac{1}{2^{n+1}})\times(0,\frac{1}{2})\subset\Om$ be the $n$'th component from right of $\Omega \cap (0,1)\times(0,1/2)$ and let $\Gamma_n$ denote the horizontal curves in $E_n$. Formula \eqref{eq:non-Radonness-formula} for $E_n$ in place of $E_1$ implies that
$$ \eta^*(\Gamma_n)=2^{n+1}\lambda(E_n)\lambda(\Om)^{\frac{1-p}{p}}=\frac{1}{2}\lambda(\Om)^{\frac{1-p}{p}}\quad\text{for all }n=2,3,\ldots. $$
We conclude that $\eta^*(\Gamma)\geq\sum_{n=1}^\infty\eta^*(\Gamma_n)=\infty$.
\end{example}
\section{Failure of the Sheaf property}\label{sec:failuresheaf}

The insights gained from studying gradient curves of $p$-harmonic functions lead us to examples exhibiting the failure of the sheaf property. The examples are normed vector spaces, and we first give a simple lemma that describes the minimal upper gradient in that setting.

Let us equip $\Rn$ with norm $\|\cdot\|$ and Lebesgue measure $\lambda$. A function $f\colon\Omega\to \R$, for $\Omega\subset \R^n$ open, is called ($L$-)Lipschitz if for all $x,y\in \Omega$ we have $|f(x)-f(y)|\leq L\|x-y\|$.
It follows that $f$ is Lipschitz with respect to the Euclidean distance too and thus, by Rademacher's theorem, differentiable almost everywhere.
Thus its differential $df_x=\sum_{i=1}^n \partial_{x_i} f(x) dx^i$ is defined almost everywhere, and it associates to a.e. $x\in \Omega$ an element $df_x$ in the dual space $(\R^n)^*$. Let $\|\cdot\|^*$ be the dual norm of $\|\cdot\|$ given by $\|\omega\|^*=\sup_{\|v\|=1} \omega(v)$ for $\omega\in (\R^n)^*$. Using these, we can compute the minimal upper gradients of Lipschitz functions in finite dimensional normed spaces.

\begin{lemma}\label{lem:minuppergradeucl} Let $\R^n$ be equipped with a metric $d$ given by a norm $\|\cdot\|$ and the Lebesgue measure, and let $f:\Omega \to \R$ be Lipschitz for some domain $\Omega \subset \R^n$. Then 
\begin{equation}\label{eq:differentialminupgrad}
g_f = \|df\|^* \text{ a.e. in } \Omega.
\end{equation}
\end{lemma}

\begin{proof}
First, if $E\subset \Omega$ is the set of points where $f$ is differentiable, then for $p$-a.e. curve $\gamma \in \Curv(\Om)$ we have $\int_\gamma 1_{\Omega \setminus E} ds=0$, cf. \cite{Bjorn2011}*{Lemma 1.42}. For such curves, one can use the chain rule to find $(f\circ\gamma)'(t) = df_{\gamma(t)}(\gamma'(t))$ for a.e. $t\in [0,\ell_\gamma]$. 
Thus
\[
|f(\gamma(\ell_\gamma))-f(\gamma(0))|\leq \int_0^{\ell_\gamma}  |(f\circ\gamma)'(t)| dt = \int_0^{\ell_\gamma}  |df_{\gamma(t)}(\gamma'(t))| ds \leq  \int_{\gamma}  \|df\|^* ds,
\]
where we used $\|\gamma'(t)\|=\lim_{h\to 0} \frac{\|\gamma(t+h)-\gamma(t)\|}{h}=|\dot{\gamma}|(t)=1$ for a.e. $t\in(0,\ell_\gamma)$, which holds since $\gamma$ is parametrized by unit speed; see 
\cite{Ambrosio2005}*{Theorem 1.1.2 and Lemma 1.1.4}. Thus, $\|df\|^*$ is a $p$-weak upper gradient for $f$. 

Next, fix $v\in \R^n$ with $\|v\|=1$. Then, for a.e.  $x\in \R^n$ , and a.e. $a,b\in \R$ with $a<b$ we have by the upper gradient property \eqref{eq:upper-gradient}
\[
|f(x+bv)-f(x+av)|\leq \int_a^b g_f(x+tv) dt.
\]
Dividing by $b-a$ and by Lebesgue differentiation one gets
\[
|df_x(v)|\leq g_f(x)\quad\text{for a.e. }x\in\Om.
\] 
Taking a supremum over a countably dense collection of $v$ with $\|v\|=1$, we get $\|df\|^*\leq g_f$. Thus, \eqref{eq:differentialminupgrad} follows.
\end{proof}

The following theorem gives the failure of the sheaf property, i.e. Theorem \ref{thm:failureofsheaf}. Recall, that in the introduction we already gave a heuristic idea for constructing the counter-example. 

\begin{theorem}\label{thm:failureofsheaf} There exists doubling metric measure space $X$ which satisfies a $p$-Poincar\'e inequality and domains $\Omega_1,\Omega_2\subset X$ and a function $u\in N^{1,p}(\Omega_1\cup\Omega_2)$ which is a $p$-harmonic energy minimizer in $\Omega_1$ and in $\Omega_2$, but is not a $p$-harmonic energy  minimizer in $\Omega_1\cup \Omega_2$. 
\end{theorem}

\begin{proof}
Consider the plane $\R^2$ equipped with the $\ell_1$-metric $d((x_1,y_1),(x_2,y_2))=|x_1-x_2|+|y_1-y_2|$
and the Lebesgue measure $\lambda$.
Then $(\R^2,d,\lambda)$ is a doubling $p$-Poincare space, since it is bi-Lipschitz to the standard Euclidean plane; see e.g. \cite{Heinonen2001}. The dual norm of the $\ell_1$-norm is the $\ell_\infty$-norm, and thus by Lemma \ref{lem:minuppergradeucl} we get $g_f = \max(|\partial_x f|,|\partial_y f|)$ for any Lipschitz function $f:\Omega\to \R$ and any domain $\Omega \subset \R^n$.

Define the following domains
\[
\Omega_1 = (-1,1)\times (0,1)\quad\text{and}\quad\Omega_2 = (0,1)\times(-1,1)
\]
and the following Lipschitz function $u:\Omega_1\cup \Omega_2 \to \R$
\[
u(x,y)=\begin{cases} y & x< 0 \\ x & y<0 \\ x+y & x,y \geq 0 \end{cases}.
\]
The function $u$ is bounded, $1$-Lipschitz, and 
\[
du = \begin{cases} dy & x< 0 \\ dx & y<0 \\ dx+dy & x,y \geq 0 \end{cases},
\]
and thus by \eqref{eq:differentialminupgrad}, we have $g_u=1$ in $\Omega_1\cup \Omega_2.$ In particular, $u\in N^{1,p}(\Omega_1\cup \Omega_2)$.

We first show that the function $u$ is a $p$-harmonic energy  minimizer in $\Omega_i$ for $i=1,2$. Let $\varphi \in N_0^{1,p}(\Omega_i)$. By definition, we can extend the function to $\varphi\in N^{1,p}(\overline{\Omega}_i)$ with $\varphi = 0$ on $\partial \Omega_i$. Consider the case $i=1$ first. For a.e. $x\in (-1,1)$, we have by Hölder's inequality and absolute continuity that:
\begin{equation}\label{eq:minimality}
\int_0^1 |\partial_y (u+\varphi)(x,y)|^p dy \geq \left|\int_0^1 \partial_y (u+\varphi)(x,y)\right|^p = |(u+\varphi)(x,1)-(u+\varphi)(x,0)|^p=1.
\end{equation}
Then, integrating this in $x$ yields
\begin{align*}
\int_{\Omega_1} g_{u+\varphi}^p d\lambda &\geq \int_{\Omega_1} |\partial_y (u+\varphi)|^p d\lambda  & \text{ Since upper gradient is greater than directional derivative. }\\
&= \int_{-1}^1 \int_0^1 |\partial_y (u+\varphi)|^p dxdy & \text{ Fubini } \\
&\geq \int_{-1}^1 \int_0^1 g_u^p dy dx& \eqref{eq:minimality} \text{ and } g_u=1\\
&= \int_{\Omega_1} g_{u}^p d\lambda.
\end{align*}
The argument for $i=2$ is symmetric. For a.e. $y\in(-1,1)$ fixed, we have by Hölder's inequality that:
\begin{equation}\label{eq:minimality2}
\int_0^1 |\partial_x (u+\varphi)(x,y)|^p dx \geq \left| \int_0^1 \partial_x (u+\varphi)(x,y)\right|^p = |(u+\varphi)(1,y)-(u+\varphi)(0,y)|^p=1.
\end{equation}
Then, integrating this in $y$ yields
\begin{align*}
\int_{\Omega_2} g_{u+\varphi}^p d\lambda &\geq \int_{\Omega_2} |\partial_x (u+\varphi)|^p d\lambda  & \text{ Since upper gradient is greater than directional derivative. }\\
&= \int_{-1}^1 \int_0^1 |\partial_x (u+\varphi)|^p dydx & \text{ Fubini } \\
&\geq \int_{-1}^1 \int_0^1 g_u^p dx dy & \eqref{eq:minimality2} \text{ and } g_u=1\\
&= \int_{\Omega_2} g_{u}^p d\lambda.
\end{align*}

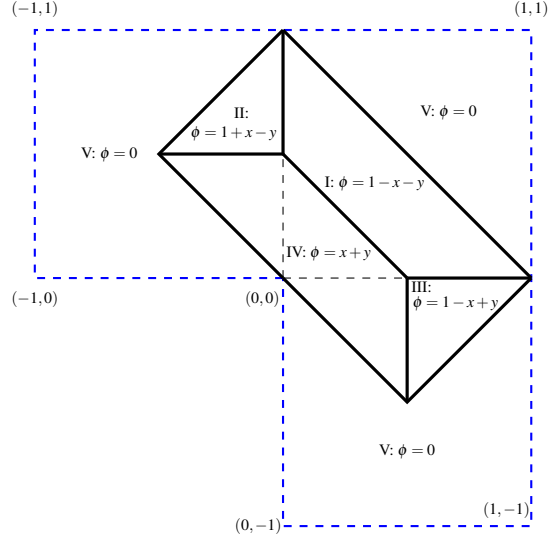
\begin{figure}[!ht]
\begin{tikzpicture}[scale=0.55]
\draw[blue,thick,dashed] (6,6) -- (0,6) -- (0,12) -- (6,12)-- (12,12) -- (12,6) -- (12,0) -- (6,0) -- (6,6); 

\node[scale=.6] at (7.1,6.6) (blabel) {IV: $\phi=x+y$};
\node[scale=.6] at (8.2,8.3) (blabel) {I: $\phi=1-x-y$};
\node[scale=.6] at (10,10) (blabel) {V: $\phi=0$};
\node[scale=.6] at (9,1.8) (blabel) {V: $\phi=0$};
\node[scale=.6] at (1.8,9) (blabel) {V: $\phi=0$};
\node[scale=.6] at (5,10) (blabel) {II:};
\node[scale=.6] at (4.8,9.5) (blabel) {$\phi=1+x-y$};
\node[scale=.6] at (9.35,5.78) (blabel) {III:}; 
\node[scale=.6] at (10.16,5.38) (blabel) {$\phi=1-x+y$}; 

\node[scale=.6] at (12,12.5) (blabel) {$(1,1)$}; 
\node[scale=.6] at (0,5.5) (blabel) {$(-1,0)$}; 
\node[scale=.6] at (0,12.5) (blabel) {$(-1,1)$}; 
\node[scale=.6] at (5.4,0) (blabel) {$(0,-1)$}; 
\node[scale=.6] at (11.4,0.4) (blabel) {$(1,-1)$}; 
\node[scale=.6] at (5.5,5.5) (blabel) {$(0,0)$}; 

\draw[thin, dashed](6,6) -- (6,9);
\draw[thin, dashed](6,6) -- (9,6); 

\draw[black, very thick](6,6) -- (3,9) -- (6,12);
\draw[black, very thick](6,9) -- (6,12);
\draw[black, very thick](9,6) -- (12,6);
\draw[black, very thick] (6,6) -- (9,3) -- (12,6);
\draw[black, very thick] (3,9) -- (6,9);
\draw[black, very thick] (9,3) -- (9,6);
\draw[black, very thick] (6,9) -- (9,6);
\draw[black, very thick] (12,6) -- (6,12); 
\end{tikzpicture}
\caption{Figure of the definition of $\phi$.}
\label{fig:regionsphi}
\end{figure}

Finally, we show that $u$ is not a $p$-harmonic energy minimizer in $\Omega_1\cup \Omega_2$. Define Lipschitz function \\
$\phi:\Omega_1\cup \Omega_2 \to \R$ with $\phi\in N^{1,p}_0(\Omega_1 \cup \Omega_2)$ piece-wise as follows:
\[
\phi(x,y) = \begin{cases} 
1-x-y & 1/2 < x+y < 1, x,y>0 \text{, case I }\\
1+x-y & x<0, y>1/2, y-x < 1 \text{, case II }\\
1-x+y & y<0, x>1/2, x-y < 1 \text{, case III }\\
x+y & 0 < x+y < 1/2, y,x<1/2 \text{, case IV }\\
0 & \text{otherwise, case V} \\
\end{cases}
\]
The regions $I,II,III,IV,V$ are defined in the cases of $\phi$, and the regions are drawn out in Figure \ref{fig:regionsphi}.  The function $\phi$ is $1$-Lipschitz in the metric $\ell_1$. 

Now, let's consider $\epsilon\in (-1/2,1/2)$ and $u_\epsilon = u+\epsilon \phi$. Then $du_\epsilon = du + \epsilon d\phi$. Therefore, by dividing into cases we compute
\[
d u_\epsilon = \begin{cases} 
(1-\epsilon)(dx+dy) & 1/2 < x+y < 1, x,y>0 \text{, case I }\\
(1-\epsilon)dy + \epsilon dx & x<0, y>1/2, y-x < 1 \text{, case II }\\
(1-\epsilon)dx + \epsilon dy & y<0, x>1/2, x-y < 1 \text{, case III }\\
(1+\epsilon)(dx+dy) & 0 < x+y < 1/2, y,x<1/2 \text{, case IV }\\
dx + dy & \text{otherwise, case V} \\
\end{cases}
\]
Thus, combining these with \eqref{eq:differentialminupgrad} we have
\[
g_{u_\epsilon} = \begin{cases} 
1-\epsilon & 1/2 < x+y < 1, x,y>0 \text{, case I }\\
1-\epsilon & x<0, y>1/2, y-x < 1 \text{, case II }\\
1-\epsilon  & y<0, x>1/2, x-y < 1 \text{, case III }\\
1+\epsilon  & 0 < x+y < 1/2, 0<y,x<1/2 \text{, case IV }\\
1 & \text{otherwise, case V}.
\end{cases}
\]
By itegrating this, we get
\[
\int_{\Omega_1\cup \Omega_2} g_{u_\epsilon}^pd\lambda 
= 2 + \frac{3}{8}(1+\epsilon)^p + \frac{5}{8}(1-\epsilon)^p.
\]
Differentiate this in $\epsilon$, and we get
\[
\frac{d}{d\epsilon}\left.\int_{\Omega_1\cup \Omega_2} g_{u_\epsilon}^p d\lambda \right|_{\epsilon = 0} = \frac{3p}{8}-\frac{5p}{8}< 0.
\]
Thus, for some small $\epsilon>0$, we have
\[
\int_{\Omega_1\cup \Omega_2} g_{u_\epsilon}^p d\lambda < \int_{\Omega_1\cup \Omega_2} g_{u}^p d\lambda. 
\]
Thus, $u$ is not a $p$-harmonic function in $\Omega_1\cup \Omega_2$.
\end{proof}
\bibliographystyle{amsplain}
\bibliography{bibliography}
\end{document}